\documentclass[a4paper]{article}
\usepackage{amsmath,amssymb,amscd,amsthm,amsfonts}
\usepackage{latexsym}
\usepackage{dsfont}
\usepackage{stmaryrd}
\usepackage{tikz-cd}
\usepackage{upgreek}
\usepackage[cal=boondox,calscaled=.96]{mathalfa}
\usepackage[colorlinks=true,linkcolor=blue,urlcolor=blue,debug]{hyperref}

\newcommand{\QQ}{\mathbb{Q}}
\newcommand{\ZZ}{\mathbb{Z}}
\newcommand{\N}{\mathbb{N}}

\newcommand{\DD}{\mathcal{D}}

\newcommand{\sbullet}{{\hspace{.1em}\scriptstyle\bullet\hspace{.1em}}}

\newcommand{\dU}{\mathds{U}}

\DeclareMathOperator{\ord}{ord}

\DeclareMathOperator{\Aut}{Aut}

\DeclareMathOperator{\supp}{\rm supp}

\DeclareMathOperator{\U}{\mathcal{U}}
\newcommand{\Ufil}{\U_{\text{\rm fil}}}

\newcommand{\bchi}{{\boldsymbol\chi}}
\newcommand{\bSigma}{\boldsymbol{\Sigma}}
\newcommand{\bepsilon}{{\boldsymbol\varepsilon}}

\newcommand{\bupchi}{{\boldsymbol\upchi}}
\newcommand{\bupupsilon}{{\boldsymbol\upupsilon}}
\newcommand{\bupvartheta}{{\boldsymbol\upvartheta}}

\DeclareMathOperator{\Sim}{Sym}
\DeclareMathOperator{\Der}{Der}
\DeclareMathOperator{\Hom}{Hom}
\DeclareMathOperator{\diff}{\mathcal{D}}
\DeclareMathOperator{\End}{End}
\DeclareMathOperator{\gr}{\rm gr}

\newcommand{\pcirc}{{\scriptstyle \,\circ\,}}

\newcommand\Id{{\rm Id}}

\newcommand{\bfs}{{\bf s}}
\newcommand{\bft}{{\bf t}}

\DeclareMathOperator{\HS}{HS}
\DeclareMathOperator{\Ider}{IDer}

\DeclareMathOperator{\Sub}{\mathcal{S}}

\DeclareMathOperator{\coideal}{\mathcal{C\!I}}
\newcommand{\coide}[1]{\coideal\left(#1\right)}

\newcommand{\fd}{\mathfrak{d}}
\DeclareMathOperator{\HH}{\Lambda}

\newcommand{\bkappa}{{\boldsymbol\kappa}}
\newcommand{\blambda}{{\boldsymbol\lambda}}

\DeclareMathOperator{\Ima}{Im}

\newtheorem{thm}{Theorem}[subsection]
\newtheorem{cor}[thm]{Corollary}
\newtheorem{prop}[thm]{Proposition}
\newtheorem{lemma}[thm]{Lemma}
\newtheorem{defi}[thm]{Definition}

\newtheorem{exam}[thm]{Example}

\newtheorem{nota}[thm]{Remark}
\newcommand{\numero}{\refstepcounter{thm}\noindent {\bf  \thethm\ }}
\newtheorem{notacion}[thm]{Notation}

\title{Hasse--Schmidt modules versus\\ integrable connections}

\author{Luis Narv\'{a}ez Macarro\thanks{Partially supported by MTM2016-75027-P, P12-FQM-2696
 and FEDER. }}

\date{}

\begin{document}

\maketitle

\begin{abstract} 

We prove that, in characteristic $0$, any Hasse-Schmidt module structure can be recovered from its underlying integrable connection, and consequently Hasse--Schmidt modules and modules endowed with an integrable connection coincide.
\bigskip

\noindent Keywords: Hasse--Schmidt derivation, integrable derivation, differential operator, substitution map, HS-structure, integrable connection.

\noindent {\sc MSC: 14F10, 13N10, 13N15.}
\end{abstract}

\section*{Introduction}

Let $k$ be a commutative ring and $A$ a commutative $k$-algebra. Let us recall what a {\em Hasse--Schmidt module} (HS-module for short) over $A/k$ is \cite[\S 3.1]{nar_envelop}.
A {\em $(p,\Delta)$-variate Hasse--Schmidt derivation} of $A$ over $k$ is a family  $D= \left(D_\alpha\right)_{\alpha\in\Delta}$ of $k$-linear endomorphisms of $A$ such that $D_0$ is the identity map and 
$$ D_\alpha (x y) = \sum_{\scriptscriptstyle \beta + \gamma = \alpha} D_\beta(x) D_\gamma(y),\quad \forall \alpha\in \Delta, \forall x,y\in A,
$$
where $\Delta \subset \N^p$ is a non-empty {\em co-ideal}, i.e. a subset of $\N^p$ such that everytime $\alpha\in \Delta$ and $\alpha'\leq \alpha$ we have $\alpha'\in \Delta$. The component $D_\alpha$ of a Hasse--Schmidt derivation $D$ is a $k$-linear differential operator of $A$ of order $\leq |\alpha|$ vanishing on the image of $k$, in particular $D_\alpha$ is a $k$-linear derivation of $A$ whenever $|\alpha|=1$.
\smallskip

We may think on Hasse--Schmidt derivations as series $D=\sum_{\scriptscriptstyle \alpha\in\Delta} D_\alpha \bfs^\alpha$ in the quotient ring $R[[\bfs]]_\Delta$ of the power series ring $R[[\bfs]] = R[[s_1,\dots,s_p]]$, $R=\End_k(A)$, by the two-sided monomial ideal generated by all $\bfs^\alpha$ with $\alpha \in \N^p \setminus \Delta$.
The set $\HS^p_k(A;\Delta)$ of $(p,\Delta)$-variate Hasse--Schmidt derivations form a subgroup of the group of units $\left(R[[\bfs]]_\Delta\right)^\times$, and they also carry an {\em action of substitution maps} \cite[\S 5]{nar_subst_2018}: 
 given a substitution map $\varphi: A[[s_1,\dots,s_p]]_\Delta \to A[[t_1,\dots,t_q]]_\nabla $ and a $(p,\Delta)$-variate Hasse--Schmidt derivation $D=\sum D_\alpha \bfs^\alpha$, a new ($q,\nabla)$-variate Hasse--Schmidt derivation is given by:
$$\varphi\sbullet D := \sum_{\scriptscriptstyle \alpha\in\Delta} \varphi(\bfs^\alpha) D_\alpha.
$$
A {\em left HS-module} over $A/k$ is an $A$-module $E$ on which Hasse--Schmidt derivations act in a compatible way with the group structure and the action of substitution maps, and satisfying a Leibniz rule. More precisely, for each $(p,\Delta)$-variate Hasse--Schmidt derivation $D=\sum D_\alpha \bfs^\alpha$ of $A$, $E$ is endowed with a $k[[\bfs]]_\Delta$-linear automorphism $\Psi^p_\Delta(D): E[[\bfs]]_\Delta \to E[[\bfs]]_\Delta$ congruent with the identity modulo $\langle \bfs\rangle$, in such a way that:
\begin{itemize}
\item[-)]
 The $\Psi^p_\Delta(-)$ are group homomorphism.
\item[-)] For each substitution map $\varphi: A[[\bfs]]_\Delta \to A[[\bft]]_\nabla $ we have $\Psi^q_\nabla(\varphi \sbullet D) = \varphi \sbullet \Psi^p_\Delta(D)$.
\item[-)] (Leibniz rule) For each $a\in A$ we have $\Psi^p_\Delta(D) a = D(a) \Psi^p_\Delta(D)$.
\end{itemize}
{\em Right HS-modules} over $A/k$ are defined in a similar way.
\smallskip

Actually, taking into account that $k[[\bfs]]_\Delta$-linear automorphisms $ E[[\bfs]]_\Delta \stackrel{\sim}{\to} E[[\bfs]]_\Delta$ congruent with the identity modulo $\langle \bfs\rangle$ can be identified with formal power series with coefficients in $\End_k(E)[\bfs]]_\Delta$ whose $0$-term is the identity, left (resp. right) HS-modules over $A/k$ appear as a particular case of the more general notion of {\em HS-structure} over $A/k$ on a $k$-algebra $S$ over $A$, for $S=\End_k(E)$ (resp. $S=\End_k(E)^{\text{\rm opp}}$).
\smallskip

Group $\HS^1_k(A;\{0,1\})$ can be identified with the additive group $\Der_k(A)$ of $k$-linear derivations of $A$, and not only the $A$-module structure on $\Der_k(A)$ is encoded into the action of substitution maps, but also the Lie bracket on $\Der_k(A)$ can be expressed in terms of the group structure of Hasse--Schmidt derivations and of the action of substitution maps.   
Consequently, 
any left (resp. right) HS-module $(E,\left\{\Uppsi^p_\Delta\right\})$ over $A/k$ carries a natural left (resp. right) integrable connection $\nabla: \Der_k(A) \to \End_k(E)$ (resp. $\nabla: \Der_k(A) \to \End_k(E)^{\text{\rm opp}}$) given by (see Corollary \ref{cor:HS-mod->IC}):
\begin{equation} \label{eq:intro-1}
  \Uppsi^1_{\{0,1\}}(\Id+\delta s) = 1 + \nabla(\delta) s,\quad \forall \delta\in \Der_k(A).
\end{equation}
This paper is devoted to prove that, whenever $\QQ \subset k$, any (left) (resp. right) integrable connection over $A/k$ on an $A$-module $E$ underlies a unique left (resp. right) HS-module structure on $E$, and so, in characteristic $0$, HS-modules and integrable connections coincide (see Corollary  \ref{cor:main}).
\smallskip

Our approach is based on the study of the map $\varepsilon$ in \cite{HS-vs-der} (see also \cite{Mirza_2010}),  associating to each Hasse--Schmidt derivation $D = \sum D_\alpha \bfs^\alpha \in \HS^p_k(A;\Delta)$ a formal power series of classical derivations
$$
\varepsilon(D) = \sum_{\substack{\scriptscriptstyle \alpha \in \Delta\\ \scriptscriptstyle \alpha \neq 0}} \varepsilon_\alpha(D) \bfs^\alpha,\quad \varepsilon_\alpha(D)\in \Der_k(A),
$$
defined as a ``logarithmic derivative''
$$
\varepsilon(D) := D^* \left(\sum_{i=1}^p s_i \frac{\partial D}{\partial s_i}\right),
$$
where $D^*$ denotes the inverse of $D$. For instance, for $D\in \HS^1_k(A;\N)$ we have:
\begin{eqnarray*}
& \varepsilon_1(D)= D_1,\ \varepsilon_2(D) = 2 D_2 - D_1^2,\ \varepsilon_3(D) = 3 D_3 - 2 D_1 D_2 - D_2 D_1 + D_1^3,
&
\\
& \varepsilon_4(D)= 4 D_4 - 3 D_1 D_3 - 2 D_2^2 - D_3 D_1 + 2 D_1^2 D_2 + D_1 D_2 D_1 + D_2 D_1^2 - D_1^4, 
\dots
\end{eqnarray*}
We prove the following results:
\begin{itemize}
\item[(I)] For each left (resp. right) HS-module $\left(E,\left\{\Uppsi^p_\Delta\right\}\right)$ over $A/k$, its underlying integrable connection $\nabla$ in (\ref{eq:intro-1}) satisfies the following compatibility with the $\varepsilon$ maps:
\begin{equation} \label{eq:intro-2}
\varepsilon(\Uppsi^p_\Delta(D)):= \Uppsi^p_\Delta(D)^* \left( \sum_{i=1}^p s_i  \frac{\partial \Uppsi^p_\Delta(D)}{\partial s_i}\right) = \sum_{\substack{\scriptscriptstyle \alpha \in \Delta\\ \scriptscriptstyle \alpha \neq 0}} \nabla (\varepsilon_\alpha(D) ) \bfs^\alpha
\end{equation}
for each non-empty co-ideal $\Delta \subset \N^p$ and each $D\in \HS^p_k(A;\Delta)$ (see Corollary \ref{cor:compatib-nabla-epsilon}).
\item[(II)] If $\QQ \subset k$ and $E$ is an $A$-module endowed with a left (resp. right) integrable connection $\nabla: \Der_k(A) \to \End_k(E)$ (resp. $\nabla: \Der_k(A) \to \End_k(E)^{\text{\rm opp}}$), there is a unique left (resp. right) HS-module structure $\left\{\Uppsi^p_\Delta\right\}$ on $E$ such that the differential equation (\ref{eq:intro-2}) holds for each non-empty co-ideal $\Delta \subset \N^p$ and each Hasse--Schmidt derivation $D\in \HS^p_k(A;\Delta)$ (see Theorem \ref{thm:from-nabla-to-Uppsi}).
\item[(III)] Let us denote $\dU_{A/k}^{\text{\rm\tiny HS}}$ the enveloping algebra of Hasse--Schmidt derivations introduced in \cite{nar_envelop} and $\dU_{A/k}^{\text{\rm\tiny LR}}$ the enveloping algebra of the Lie-Rinehart algebra $\Der_k(A)$ (see \cite[\S 2]{rine-63}). There is a canonical map of filtered $k$-algebras
$ \bkappa: \dU_{A/k}^{\text{\rm\tiny LR}} \to \dU_{A/k}^{\text{\rm\tiny HS}}
$
which is an isomorphism provided that $\QQ \subset k$ (see Theorem \ref{thm:bkappa-iso}).
\end{itemize}
\medskip

Let us now comment on the content of this paper.
\medskip

Section 1 contains the notions and notations used in the paper. We recall the construction and the main properties of the $\varepsilon$ maps in \cite[\S 1.2]{HS-vs-der}, the action of substitution maps on Hasse--Schmidt derivations \cite[\S 5]{nar_subst_2018} and the behavior of $\varepsilon$ under substitution maps (Theorem 3.2.5 of \cite{HS-vs-der}).
\medskip

Section 2 contains the main results of the paper. First, we recall the notion of HS-structure on a $k$-algebra over $A$, the notions of left and right HS-modules over $A/k$, and the existence and main properties of the enveloping algebra of Hasse--Schmidt derivations \cite[\S 3.3]{nar_envelop}. Second, we prove (I), (II) and (III) above.

\section{Notations and preliminaries}

\subsection{Notations}

Throughout the paper we will use the following notations:
\medskip

\noindent $k$ is a commutative ring and $A$ a commutative
$k$-algebra.
\smallskip

\noindent $R$, $S$ are not-necessarily commutative rings, often $k$-algebras (over $A$,  Definition \ref{def:k-algebra-over-A}).

\noindent $\diff_{A/k}$: the ring of $k$-linear differential operators of $A$, \cite{ega_iv_4}.
\smallskip

\noindent $\bfs = \{s_1,\dots,s_p\}$, $\bft =\{t_1,\dots,t_q\}$, \dots\ are sets of variables.
\smallskip

\noindent $\coide{\N^p}$: the set of all non-empty co-ideals of $\N^p$, Definition 
\ref{def:ideal-co-ideal}.
\smallskip

\noindent $M[[\bfs]]_\Delta, M[[\bfs]]_{\Delta,+}$:  \ref{nume:M[[bfs]]Delta}.
\smallskip

\noindent $\U^p(R;\Delta), \Ufil^p(R;\Delta)$: Notation \ref{notacion:Ump}.
\smallskip

\noindent $\Sub_A(p,q;\Delta,\nabla)$: the set of substitution maps $A[[\bfs]]_\Delta \to A[[\bft]]_\nabla$, Definition \ref{def:substitution-maps}.
\smallskip

\noindent $\varphi \sbullet r, r\sbullet \varphi$: \ref{nume:def-sbullet}.
\smallskip

\noindent $\HS^p_k(A;\Delta)$: the group of $(p,\Delta)$-variate Hasse--Schmidt derivations, Definition \ref{def:HS}.
\smallskip

\noindent $\Ider^f_k(A)$: the module of f-integrable derivations, Definition \ref{def:HS-integ}.
\smallskip

\noindent $\Gamma_A M$: the universal power divided algebra of the $A$-module $M$, endowed with the power divided maps $\gamma_m: M \to \Gamma_A M$, cf. \cite[Appendix A]{bert_ogus}. 

\subsection{Some constructions on power series rings and modules}

Throughout this section, $k$ will be a commutative ring, $A$ a commutative $k$-algebra and $R$ a ring, not-necessarily commutative.
\medskip

Let $p\geq 1$ be an integer and let us call $\bfs = \{s_1,\dots,s_p\}$ a set of $p$ variables. The support of each $\alpha\in \N^p$ is defined as  $\supp \alpha := \{i\ |\ \alpha_i\neq 0\}$. 
The monoid $\N^p$ is endowed with a natural partial ordering. Namely, for $\alpha,\beta\in \N^p$, we define
$$ \alpha \leq \beta\quad \stackrel{\text{def.}}{\Longleftrightarrow}\quad \exists \gamma \in \N^p\ \text{such that}\ \beta = \alpha + \gamma\quad \Longleftrightarrow\quad \alpha_i \leq \beta_i\quad \forall i=1\dots,p.$$
We denote $|\alpha| := \alpha_1+\cdots+\alpha_p$. 
\medskip

If $M$ is an abelian group and $M[[\bfs]]$ is the abelian group of power series with coefficients in $M$,
the {\em support}  of a series $m=\sum_\alpha m_\alpha \bfs^\alpha \in M[[\bfs]]$ is $\supp(m) := \{  \alpha \in \N^p\ |\  m_\alpha \neq 0\} \subset \N^p$. We have $m=0\Leftrightarrow \supp(m) = \emptyset$. 
The abelian group $M[[\bfs]]$ is clearly a $\ZZ[[\bfs]]$-module, which will be always endowed with the $\langle \bfs \rangle$-adic topology.

\begin{defi} \label{def:ideal-co-ideal}
We say that a subset $\Delta \subset \N^p$ is  a
{\em co-ideal} of $\N^p$  
if everytime $\alpha\in\Delta$ and $\alpha'\leq \alpha$, then $\alpha'\in\Delta$.
The set of all non-empty co-ideals of $\N^p$ will be denoted by $\coide{\N^p}$. 
\end{defi}
\bigskip

\numero \label{nume:M[[bfs]]Delta}
Let $M$ be an abelian group. 
For each co-ideal $\Delta \subset \N^p$, we denote by $\Delta_M$ the closed sub-$\ZZ[[\bfs]$-bimodule of $M[[\bfs]]$ whose elements are the formal power series $\sum_{\alpha\in\N^p} m_\alpha \bfs^\alpha$ such that $m_\alpha=0$ whenever $\alpha\in \Delta$, and $M[[\bfs]]_\Delta := M[[\bfs]]/\Delta_M$. The elements in $M[[\bfs]]_\Delta$ are power series of the form 
$\sum_{\scriptscriptstyle \alpha\in\Delta} m_\alpha \bfs^\alpha$, $m_\alpha \in M$. If $f:M\to M'$ is a homomorphism of abelian groups, we will denote by $\overline{f}: M[[\bfs]]_\Delta \to M'[[\bfs]]_\Delta$ the $\ZZ[[\bfs]]_\Delta$-linear map defined as $\overline{f}\left(\sum_{\scriptscriptstyle \alpha\in\Delta} m_\alpha \bfs^\alpha\right) = \sum_{\scriptscriptstyle \alpha\in\Delta} f(m_\alpha) \bfs^\alpha$.
\medskip

If $R$ is a ring, then $\Delta_R$ is a closed two-sided ideal of $R[[\bfs]]$ and so $R[[\bfs]]_\Delta$ is a topological ring, which we always consider endowed with the $\langle \bfs \rangle$-adic topology ($=$ to the quotient topology). Similarly, if $M$ is an $(A;A)$-bimodule (central over $k$), then $M[[\bfs]]_\Delta$ is an $(A[[\bfs]_\Delta;A[[\bfs]]_\Delta)$-bimodule (central over $k[[\bfs]]_\Delta$).
\medskip

For $\Delta' \subset \Delta$ non-empty co-ideals of $\N^p$, we have natural $\ZZ[[\bfs]]$-linear projections 
$\tau_{\Delta \Delta'}:M[[\bfs]]_{\Delta} \longrightarrow M[[\bfs]]_{\Delta'}$, that we call {\em truncations}:
$$\tau_{\Delta \Delta'} : \sum_{\scriptscriptstyle \alpha\in\Delta} m_\alpha \bfs^\alpha \in M[[\bfs]]_{\Delta'} \longmapsto \sum_{\scriptscriptstyle \alpha\in\Delta'} m_\alpha \bfs^\alpha \in M[[\bfs]]_{\Delta}.
$$
If $M$ is a ring (resp. an $(A;A)$-bimodule), then the truncations $\tau_{\Delta \Delta'}$ are ring homomorphisms (resp.  $(A[[\bfs]]_{\Delta};A[[\bfs]]_{\Delta})$-linear maps). For $\Delta' = \{0\}$ we have $ M[[\bfs]]_{\Delta'} = M$ and the kernel of $\tau_{\Delta \{0\}}$ will be denoted by $M[[\bfs]]_{\Delta,+}$. 
We have a bicontinuous isomorphism
$\displaystyle M[[\bfs]]_\Delta = \lim_{\longleftarrow} M[[\bfs]]_{\Delta'}$, 
where $\Delta'$ goes through the set of finite co-ideals contained in $\Delta$.

\begin{defi} \label{def:k-algebra-over-A}
A $k$-algebra over $A$  is a (not-necessarily commutative) $k$-algebra $R$ endowed with a map of $k$-algebras $\iota:A \to R$. A filtered $k$-algebra over $A$  is a $k$-algebra $(R,\iota)$ over $A$, endowed with a ring filtration $(R_k)_{k\geq 0}$ such that $\iota(A) \subset R_0$. A map between two (filtered) $k$-algebras $\iota:A \to R$ and $\iota':A \to R'$ over $A$ is a map $g:R\to R'$ of (filtered) $k$-algebras such that $\iota' = g \pcirc \iota$. 
\end{defi}

It is clear that if $R$ is a $k$-algebra over $A$, then $R[[\bfs]]_\Delta$ is a $k[[\bfs]]_\Delta$-algebra over $A[[\bfs]]_\Delta$.

\begin{notacion} \label{notacion:Ump} Let $R$ be a ring, $p\geq 1$ and  $\Delta\subset \N^p$ a non-empty co-ideal. We denote by $\U^p(R;\Delta)$ the multiplicative sub-group of the units of $R[[\bfs]]_\Delta$ whose 0-degree coefficient is $1$. The multiplicative inverse of a unit $r\in R[[\bfs]]_\Delta$ will be denoted by $r^*$.
For $\Delta\subset \Delta'$ co-ideals we have $\tau_{\Delta'\Delta}\left(\U^p(R;\Delta')\right) \subset \U^p(R;\Delta)$ and the truncation map
$\tau_{\Delta'\Delta}:\U^p(R;\Delta') \to \U^p(R;\Delta)$ is a group homomorphisms. Clearly, we have:
\begin{equation} \label{eq:U-inv-limit-finite}
   \U^p(R;\Delta)  
   = \lim_{\stackrel{\longleftarrow}{\substack{\scriptscriptstyle \Delta' \subset \Delta\\ \scriptscriptstyle  \sharp \Delta'<\infty}}} \U^p(R;\Delta').
\end{equation}

\noindent
If $R = \cup_{d\geq 0} R_d$ is a filtered ring, we denote: 
$$\Ufil^p(R;\Delta):= \left\{ \left. \sum_{\scriptscriptstyle \alpha\in\Delta} r_\alpha \bfs^\alpha\in\U^p(R;\Delta)\  \right\vert \  r_\alpha \in R_{|\alpha|}\ \forall \alpha \in\Delta\right\}.
$$
It is clear that $\Ufil^p(R;\Delta)$ is a subgroup of $\U^p(R;\Delta)$.
For any ring homomorphism $f:R\to R'$, the induced ring homomorphism $\overline{f}: R[[\bfs]]_\Delta \to R'[[\bfs]]_\Delta$ sends $\U^p(R;\Delta)$ into $\U^p(R';\Delta)$ and so it induces natural group homomorphisms
$\U^p(R;\Delta) \to \U^p(R';\Delta)$. A similar result holds for the filtered case.
\end{notacion}

\numero \label{nume:widetilde-r}
 Let $E,F$ be $A$-modules. For each $r= \sum_\beta r_\beta \bfs^\beta \in \Hom_k(E,F)[[\bfs]]_\Delta$ we denote by $\widetilde{r}: E[[\bfs]]_\Delta \to F[[\bfs]]_\Delta$ the map defined by:
$$ \widetilde{r} \left( \sum_{\scriptscriptstyle \alpha\in \Delta} e_\alpha \bfs^\alpha \right) :=
\sum_{\scriptscriptstyle \alpha\in \Delta } \left( \sum_{\scriptscriptstyle \beta + \gamma=\alpha} r_\beta(e_\gamma) \right) \bfs^\alpha,
$$
which is obviously a $k[[\bfs]]_\Delta$-linear map. 
It is clear that the map:
\begin{equation} \label{eq:tilde-map}
 r\in \Hom_k(E,F)[[\bfs]]_\Delta \longmapsto \widetilde{r} \in \Hom_{k[[\bfs]]_\Delta}(E[[\bfs]]_\Delta,F[[\bfs]]_\Delta)
\end{equation}
is an isomorphism of $(A[[\bfs]]_\Delta;A[[\bfs]]_\Delta)$-bimodules. When $E=F$, it is an isomorphism of 
$k[[\bfs]]_\Delta$-algebras over $A[[\bfs]]_\Delta$. Moreover, the restriction map:
$$f \in  \Hom_{k[[\bfs]]_\Delta}(E[[\bfs]]_\Delta,F[[\bfs]]_\Delta) \longmapsto f|_E \in \Hom_k(E,F[[\bfs]]_\Delta)
$$
is an isomorphism of $(A[[\bfs]]_\Delta;A)$-bimodules, whose inverse is
$$ g\in \Hom_k(E,F[[\bfs]]_\Delta) \mapsto g^e \in \Hom_{k[[\bfs]]_\Delta}(E[[\bfs]]_\Delta,F[[\bfs]]_\Delta),\  g^e \left( \sum_\alpha e_\alpha \bfs^\alpha \right) := \sum_\alpha g(e_\alpha) \bfs^\alpha.   
$$
Let us call $R = \End_k(E)$. As a consequence of the above properties, the composition of the maps:
\begin{equation}  \label{eq:comple-formal-1}
 R[[\bfs]]_\Delta \xrightarrow{r \mapsto \widetilde{r}} \End_{k[[\bfs]]_\Delta}(E[[\bfs]]_\Delta) \xrightarrow{f \mapsto f|_E} \Hom_k(E,E[[\bfs]]_\Delta)
\end{equation}
is an isomorphism of $(A[[\bfs]]_\Delta;A)$-bimodules. 

\begin{notacion} \label{notacion:pcirc}
We denote:
$$
\Hom_k^\pcirc(E,E[[\bfs]]_\Delta) := \left\{ f \in \Hom_k(E,E[[\bfs]]_\Delta), f(e) \equiv e\!\!\!\!\mod \langle \bfs \rangle E[[\bfs]]_\Delta\ \forall e\in E \right\}, 
$$
$$
\Aut_{k[[\bfs]]_\Delta}^\pcirc(E[[\bfs]]_\Delta) 
:= \left\{    
f \in \Aut_{k[[\bfs]]_\Delta}(E[[\bfs]]_\Delta), f(e) \equiv e_0\, \text{\rm mod} \langle \bfs \rangle E[[\bfs]]_\Delta\ \forall e\in E[[\bfs]]_\Delta  \right\}.
$$
Let us notice that a $f \in \Hom_k(E,E[[\bfs]]_\Delta)$, given by $f(e) = \sum_{\alpha\in \Delta} f_\alpha(e) \bfs^\alpha$, belongs to $\Hom_k^\pcirc(E,E[[\bfs]]_\Delta)$ if and only if $f_0=\Id_E$.
\end{notacion}

The isomorphism in (\ref{eq:comple-formal-1}) gives rise to a group isomorphism
\begin{equation} \label{eq:U-iso-pcirc}
 r\in \U^p(\End_k(E);\Delta) \stackrel{\sim}{\longmapsto} \widetilde{r} \in \Aut_{k[[\bfs]]_\Delta}^\pcirc(E[[\bfs]]_\Delta). 
\end{equation}
\bigskip

If $R$ is a (not necessarily commutative) $k$-algebra and $\Delta \subset \N^p$ is a co-ideal, any continuous $k$-linear map $h: k[[\bfs]]_\Delta \to k[[\bfs]]_\Delta$ induces a natural continuous left and right $R$-linear map 
$$ h_R:= \Id_R \widehat{\otimes}_k h:R[[\bfs]]_\Delta = R \widehat{\otimes}_k k[[\bfs]]_\Delta \longrightarrow R[[\bfs]]_\Delta = R \widehat{\otimes}_k k[[\bfs]]_\Delta 
$$
given by 
$$ h_R\left(\sum_\alpha r_\alpha \bfs^\alpha \right) = \sum_\alpha r_\alpha h(\bfs^\alpha).
$$
If $\fd:k[[\bfs]]_\Delta \to k[[\bfs]]_\Delta$ is a $k$-derivation, it is continuous and $\fd_R: R[[\bfs]]_\Delta \to R[[\bfs]]_\Delta$ is a  $(R;R)$-linear derivation, i.e.:
$$ \fd_R(sr)=s\fd_R(r),\ \fd_R(rs)=\fd_R(r)s,\ \fd_R( r r') =  \fd_R ( r ) r' + r \fd_R( r'),\ \forall s\in R,\ \forall  r,r'\in R[[\bfs]]_\Delta.
$$
The following definition provides a particular family of $k$-derivations.

\begin{defi} \label{def:Euler-derivation}
For each $i=1,\dots,p$, 
the {\em $i$th partial Euler $k$-derivation}  is $\chi^i = s_i\frac{\partial}{\partial s_i}:k[[\bfs]] \to k[[\bfs]]$. It induces a $k$-derivation on each $k[[\bfs]]_\Delta$, which will be also denoted by $\chi^i$.
\medskip

\noindent 
The {\em Euler $k$-derivation} $\bchi:k[[\bfs]] \to k[[\bfs]]$
is defined as:
$$ \bchi = \sum_{i=1}^p \chi^i,\quad
\bchi\left( \sum_{\scriptscriptstyle \alpha} c_\alpha \bfs^\alpha \right) =  \sum_{\scriptscriptstyle \alpha} |\alpha| c_\alpha \bfs^\alpha.
$$
It induces a $k$-derivation on each $k[[\bfs]]_\Delta$, which will be also denoted by $\chi$.
\end{defi}

\begin{notacion} \label{notacion:varepsilon-fd}
For any $k$-derivation $\fd:k[[\bfs]]_\Delta \to k[[\bfs]]_\Delta$ and for any $r\in \U^p(R;\Delta)$, we denote:
$$\varepsilon^\fd(r) :=r^* \fd_R(r),\quad \overline{\varepsilon}^\fd(r) :=\fd_R(r) r^*,
$$
and we will write:
$$ \varepsilon^\fd(r)= \sum_{\alpha} \varepsilon^\fd_\alpha(r) \bfs^\alpha,\quad 
 \overline{\varepsilon}^\fd(r)= \sum_{\alpha} \overline{\varepsilon}^\fd_\alpha(r) \bfs^\alpha.
$$
We will simply denote:
\begin{itemize}
\item[-)] $ \varepsilon^i(r) := \varepsilon^\fd(r)$, $\overline{\varepsilon}^i(r) :=  \overline{\varepsilon}^\fd(r)
$ if $\fd= \chi^i$, $i=1,\dots,p$.
\item[-)]  $\varepsilon(r) := \varepsilon^\fd(r)$, $\overline{\varepsilon}(r) :=  \overline{\varepsilon}^\fd(r)$ if $\fd=\bchi$.
\end{itemize}
Clearly, $ \varepsilon = \sum_{i=1}^p \varepsilon^i$ and $\overline{\varepsilon} = \sum_{i=1}^p \overline{\varepsilon}^i$.
\end{notacion}
\bigskip

\numero \label{nume:explicit-varepsilon}  For the ease of the reader, here we collect several results of
\cite{HS-vs-der} (see Lemma 1.2.13, 1.2.14 and Lemma 1.2.16 in {\em loc.~cit.}).
\medskip

Let $\fd, \fd':k[[\bfs]]_\Delta \to k[[\bfs]]_\Delta$ be 
 $k$-derivations, $r,r'\in\U^p_k(R;\Delta)$ and $i,j=1,\dots,p$. Then, the following identities hold:
\begin{enumerate}
\item[(i)] 
$\overline{\varepsilon}^\fd(1) = \varepsilon^\fd(1) =0$, $\varepsilon^\fd(r'\, r) = \varepsilon^\fd(r) + r^*\, \varepsilon^\fd(r')\, r$, 
$\overline{\varepsilon}^\fd(r\, r') = \overline{\varepsilon}^\fd(r) + r\, \overline{\varepsilon}^\fd(r')\, r^*$.
\item[(ii)]
$\varepsilon^\mathfrak{d}(r^*) = -r\, \varepsilon^\fd(r)\, r^* = -\overline{\varepsilon}^\fd(r)$.
\item[(iii)]
$\varepsilon^{[\fd,\fd']}(r) = \left[\varepsilon^\mathfrak{d}(r),\varepsilon^{\fd'}(r)\right] + \fd_R \left( \varepsilon^{\fd'}(r) \right) - \fd'_R \left( \varepsilon^\fd(r) \right) $.
\item[(iv)] $ \displaystyle \varepsilon^i(r)=  \sum_{\scriptscriptstyle \alpha} \left(\sum_{\scriptscriptstyle \beta+\gamma=\alpha} \gamma_i r^*_\beta \, r_\gamma \right) \bfs^\alpha$, 
$ \displaystyle 
\varepsilon(r)= \sum_{\scriptscriptstyle \alpha} \left(\sum_{\scriptscriptstyle \beta+\gamma=\alpha} |\gamma| r^*_\beta \, r_\gamma \right) \bfs^\alpha$,

$ \displaystyle \overline{\varepsilon}^i(r)=  \sum_{\scriptscriptstyle \alpha} \left(\sum_{\scriptscriptstyle \beta+\gamma=\alpha} \beta_i r_\beta \, r^*_\gamma \right) \bfs^\alpha$, 
$ \displaystyle 
\overline{\varepsilon}(r)= \sum_{\scriptscriptstyle \alpha} \left(\sum_{\scriptscriptstyle \beta+\gamma=\alpha} |\beta| r_\beta \, r^*_\gamma \right) \bfs^\alpha$.\\ In particular, by writing $\varepsilon^i(r)= \sum_{\alpha} \varepsilon^i_\alpha(r) \bfs^\alpha$, $\overline{\varepsilon}^i(r)= \sum_{\alpha} \overline{\varepsilon}^i_\alpha(r) \bfs^\alpha$, $\varepsilon(r)= \sum_{\alpha} \varepsilon_\alpha(r) \bfs^\alpha$ and $\overline{\varepsilon}(r)= \sum_{\alpha} \overline{\varepsilon}_\alpha(r) \bfs^\alpha $, we have $\varepsilon^i_\alpha(r) = \overline{\varepsilon}^i_\alpha(r) =0$ whenever $\alpha_i=0$, i.e. whenever $i\notin \supp \alpha$, and $\varepsilon_0(r) = \overline{\varepsilon}_0(r) =0$ and so $\varepsilon^i(r), \overline{\varepsilon}^i(r), \varepsilon(r), \overline{\varepsilon}(r) \in R[[\bfs]]_{\Delta,+}$ (see (\ref{nume:M[[bfs]]Delta})).
\item[(v)]
$ \chi^j_R \left( \varepsilon^i(r) \right) -
\chi^i_R \left( \varepsilon^j(r) \right) = [\varepsilon^i(r),\varepsilon^j(r)]$, and a similar identity holds for the $\overline{\varepsilon}^i$.
\end{enumerate}

\begin{notacion} \label{notacion:Lambda-bepsilon}
Under the above conditions, we will denote by $ \HH^p(R;\Delta)$ the subset of $\left( R[[\bfs]]_{\Delta,+} \right)^p$ whose elements are the families $\{\delta^i\}_{1\leq i\leq p}$ satisfying the following properties:
\begin{enumerate}
\item[(a)] If $\delta^i = \sum_{|\alpha|>0} \delta^i_\alpha \bfs^\alpha$, we have $\delta^i_\alpha=0$ whenever $\alpha_i=0$.
\item[(b)] For all $i,j=1,\dots,p$ we have $\chi^j_R \left( \delta^i \right) -
\chi^i_R \left( \delta^j \right) = [\delta^i,\delta^j]$.
\end{enumerate}
Let us also consider the map
$ \bSigma: \{\delta^i\} \in \HH^p(R;\Delta) \longmapsto \sum_{i=1}^p \delta^i \in R[[\bfs]]_{\Delta,+}$. 
After  \ref{nume:explicit-varepsilon}, (v), we can consider the map:
$$ \bepsilon: r \in \U^p(R;\Delta) \longmapsto \{ \varepsilon^i(r)\}_{1\leq i\leq p} \in \HH^p(R;\Delta)
$$
and we obviously have $\varepsilon = \bSigma \pcirc \bepsilon$.
\end{notacion}

The following statement reproduces Proposition 1.2.18 of \cite{HS-vs-der}.

\begin{prop} \label{prop:varepsilon-bijective-Q}
Assume that $\QQ \subset k$. Then, the three maps in the following commutative diagram:
$$
\begin{tikzcd}
\U^p(R;\Delta)   
 \ar[]{r}{\bepsilon} \ar{dr}{\varepsilon } & \HH^p(R;\Delta) \ar[]{d}{\bSigma} 
 \\
 & R[[\bfs]]_{\Delta,+}
\end{tikzcd}
$$
are bijective.
\end{prop}

Notice that Proposition \ref{prop:varepsilon-bijective-Q} also holds with the $\overline{\varepsilon}^i$ instead of the $\varepsilon^i$.

\subsection{Substitution maps}

In this section we give a summary of sections 2 and 3 of \cite{nar_subst_2018}. Let $k$ be a commutative ring, $A$ a commutative $k$-algebra, $\bfs=\{s_1,\dots,s_p\},\bft=\{t_1,\dots,t_q\}$ two sets of variables and
$\Delta\subset \N^p, \nabla\subset \N^q$ non-empty co-ideals.
\medskip

\begin{defi}  \label{def:substitution-maps} 
An $A$-algebra map $\varphi:A[[\bfs]]_\Delta \xrightarrow{} A[[\bft]]_\nabla$
will be called a {\em substitution map} whenever $\ord (\varphi(s_i)) \geq 1$ for all $i=1,\dots, p$.
A such map is continuous and uniquely determined by the family $c=\{\varphi(s_i), i=1,\dots, p\}$.
The set of substitution maps $A[[\bfs]]_\Delta \xrightarrow{} A[[\bft]]_\nabla$ will be denoted by $\Sub_A(p,q;\Delta,\nabla)$.
\end{defi}

The composition of substitution maps is obviously a substitution map.

\begin{defi} We say that a substitution map $\varphi: A[[\bfs]]_\Delta \xrightarrow{} A[[\bft]]_\nabla $ has {\em constant coefficients} 
if $\varphi(s_i) \in k[[\bft]]_\nabla$ for all $i=1,\dots,p$.
\end{defi}

Substitution maps with constant coefficients are induced by substitution maps $k[[\bfs]]_\Delta \xrightarrow{} k[[\bft]]_\nabla$.
\bigskip

\numero \label{nume:def-sbullet}
Let $R$ be a $k$-algebra over $A$ and $\varphi: A[[\bfs]]_\Delta \xrightarrow{} A[[\bft]]_\nabla $ a substitution map. For $r=\sum_\alpha
r_\alpha \bfs^\alpha \in R[[\bfs]]_\Delta$ we denote:
\begin{eqnarray*}
& \displaystyle  
\varphi \sbullet r := \sum_\alpha \varphi(\bfs^\alpha)
r_\alpha,\quad
 r \sbullet \varphi 
 := \sum_\alpha 
r_\alpha \varphi(\bfs^\alpha).
\end{eqnarray*}
It is clear that $\varphi \sbullet \U^p(R;\Delta) \subset \U^q(R;\nabla)$ and $\U^p(R;\Delta)\sbullet \varphi \subset \U^q(R;\nabla)$,
and if $R$ is a filtered $k$-algebra over $A$, then  $\varphi \sbullet \Ufil^p(R;\Delta) \subset \Ufil^q(R;\nabla)$ and $\Ufil^p(R;\Delta)\sbullet \varphi \subset \Ufil^q(R;\nabla)$.
We also have $\varphi \sbullet 1 = 1 \sbullet \varphi = 1$. 
\medskip

\noindent
If $\varphi$ is a substitution map with \underline{constant coefficients}, then  $\varphi \sbullet r = r \sbullet \varphi$ and $\varphi \sbullet (rr') = (\varphi \sbullet r) (\varphi \sbullet r')$. 
Additional information about the $\sbullet$ operations can be found at \cite[\S 4]{nar_subst_2018}.
\bigskip

Let us consider the power series ring $A[[\bfs,\tau]] = A[[\bfs]] \widehat{\otimes}_A A[[\tau]]$, and for each
 $i=1,\dots,p$ we denote $\sigma^i: A[[\bfs]] \to A[[\bfs,\tau]]$ the substitution map (with constant coefficients) defined by:
$$ \sigma^i(s_j) = \left\{ \begin{array}{lll}
s_i+s_i \tau & \text{if} & j=i\\
s_j & \text{if} & j\neq i.
\end{array} \right.
$$
Let us also denote $\sigma: A[[\bfs]] \to A[[\bfs,\tau]]$ the substitution map (with constant coefficients) defined by:
$ \sigma(s_i) = s_i + s_i \tau$  for all $i=1,\dots,p$, 
and $\iota: A[[\bfs]] \to A[[\bfs,\tau]]$ the substitution map induced by the inclusion $\bfs \hookrightarrow \bfs \sqcup \{\tau\}$.
\medskip

It is clear that for each non-empty co-ideal $\Delta \subset \N^p$, the substitution maps $\sigma^i,\sigma, \iota: A[[\bfs]] \to A[[\bfs,\tau]]$ induce new substitution maps $A[[\bfs]]_\Delta \to A[[\bfs,\tau]]_{\Delta \times \{0,1\}}$, which will be also denoted by the same letters. 
\medskip

The proof of the following lemma is clear. 

\begin{lemma} \label{lemma:xi}
The map $\xi: R[[\bfs]]_{\Delta,+} \to \U^{p+1}(R;\Delta \times \{0,1\})$  defined as:
$$ \xi \left( \sum_{\scriptscriptstyle \alpha\in\Delta, |\alpha|>0} r_\alpha \bfs^\alpha \right) =
1 + \sum_{\scriptscriptstyle \alpha\in\Delta, |\alpha|>0} r_\alpha \bfs^\alpha \tau
$$
is a group homomorphism. Moreover, the map $\xi$ is injective and its image is the set of $r\in \U^{p+1}(R;\Delta \times \{0,1\})$ such that $\supp r \subset \{(0,0)\} \cup ((\Delta \setminus \{0\}) \times \{1\})$.
\end{lemma}

The following proposition is proved in \cite[Proposition 1.3.7]{HS-vs-der}.

\begin{prop} \label{prop:xi-varepsilon}
For each $r\in \U^p(R;\Delta)$, the following properties hold:
\begin{enumerate}
\item[(1)] $ r^* (\sigma^i\sbullet r)  = \xi( \varepsilon^i(r))$, $  (\sigma^i\sbullet r) r^*  = \xi( \overline{\varepsilon}^i(r))$.
\item[(2)] $ r^* (\sigma\sbullet r)  = \xi( \varepsilon(r))$, $ (\sigma\sbullet r) r^*   = \xi( \overline{\varepsilon}(r))$.
\end{enumerate}
\end{prop}

The following lemma shows how the bracket of two elements of $R$ can be expressed in terms of the group operation in the $\U^p(R;\Delta)$ and of the action of substitution maps. Its proof is straightforward and it is left to the reader.

\begin{lemma} \label{lema:aux-brackets}
Let $\iota:A[[s]]_1 \to A[[s,s']]_{(1,1)}$, $\iota':A[[s]]_1 \to A[[s,s']]_{(1,1)}$ and $\varphi:A[[s]]_1 \to A[[s,s']]_{(1,1)}$ the substitution maps (with constant coefficients) given by $\iota(s) = s$, $\iota'(s)= s'$ and $\varphi(s)=s s'$. Then, for each $r,r'\in R$ we have:
$$
\left(\iota\sbullet (1 + r s) \right) \left( \iota'\sbullet (1 + r' s) \right) \left(\iota\sbullet (1 + r s)\right)^* \left( \iota'\sbullet (1 + r' s) \right)^* = \varphi\sbullet (1 + [r,r']s).
$$
\end{lemma}

\subsection{Hasse--Schmidt derivations}
\label{section:HS}

In this section we recall some notions and results of the theory of Hasse--Schmidt derivations \cite{has37} as developed in \cite{nar_subst_2018}. 
From now on $k$ will be a commutative ring, $A$ a commutative $k$-algebra, $\bfs =\{s_1,\dots,s_p\}$ a set of variables and $\Delta \subset \N^p$ a non-empty co-ideal.

\begin{defi} \label{def:HS}
A {\em $(p,\Delta)$-variate Hasse--Schmidt derivation}, or 
a {\em $(p,\Delta)$-va\-riate HS-de\-ri\-va\-tion} for short, of $A$ over $k$  
 is a family $D=(D_\alpha)_{\alpha\in \Delta}$ 
 of $k$-linear maps $D_\alpha:A
\longrightarrow A$, satisfying the following Leibniz type identities: $$ D_0=\Id_A, \quad
D_\alpha(xy)=\sum_{\beta+\gamma=\alpha}D_\beta(x)D_\gamma(y) $$ for all $x,y \in A$ and for all
$\alpha\in \Delta$. We denote by
$\HS^p_k(A;\Delta)$ the set of all $(p,\Delta)$-variate HS-derivations of $A$ over
$k$. For $p=1$, a $1$-variate HS-derivation will be simply called a {\em Hasse--Schmidt derivation}   (a HS-derivation for short), or a {\em higher derivation}\footnote{This terminology is used for instance in \cite[\S 27]{mat_86}.}, and we will simply write $\HS_k(A;m):= \HS^1_k(A;\Delta)$ for $\Delta=\{q\in\N\ |\ q\leq m\}$.\footnote{These HS-derivations are called of length $m$ in \cite[\S 27]{mat_86}.}
\end{defi}

Any $(p,\Delta)$-variate HS-derivation $D$ of $A$ over $k$  can be understood as a power series 
$$ \sum_{\scriptscriptstyle \alpha\in\Delta} D_\alpha \bfs^\alpha \in R[[\bfs]]_\Delta,\quad R=\End_k(A),$$
and so we consider $\HS^p_k(A;\Delta) \subset R[[\bfs]]_\Delta$. Actually, 
$\HS^p_k(A;\Delta)$ is a (multiplicative) sub-group of $\U^p(R;\Delta)$. 
The group operation in $\HS^p_k(A;\Delta)$ is explicitly given by:
$$ (D,E) \in \HS^p_k(A;\Delta) \times \HS^p_k(A;\Delta) \longmapsto D \pcirc E \in \HS^p_k(A;\Delta)$$
with
$$ (D \pcirc E)_\alpha = \sum_{\scriptscriptstyle \beta+\gamma=\alpha} D_\beta \pcirc E_\gamma,
$$
and the identity element of $\HS^p_k(A;\Delta)$ is $\mathbb{I}$ with $\mathbb{I}_0 = \Id$ and 
$\mathbb{I}_\alpha = 0$ for all $\alpha \neq 0$. The inverse of a $D\in \HS^p_k(A;\Delta)$ will be denoted by $D^*$. 
\medskip

For $\Delta' \subset \Delta \subset \N^p$ non-empty co-ideals, we have truncations
$$\tau_{\Delta \Delta'}: \HS^p_k(A;\Delta) \longrightarrow \HS^p_k(A;\Delta'),
$$ which obviously are group homomorphisms. 
Since any $D\in\HS_k^p(A;\Delta)$ is determined by its finite truncations, we have a natural group isomorphism
\begin{equation} \label{eq:HS-inv-limit-finite}
 \HS_k^p(A;\Delta) =  \lim_{\stackrel{\longleftarrow}{\substack{\scriptscriptstyle \Delta' \subset \Delta\\ \scriptscriptstyle  \sharp \Delta'<\infty}}} \HS_k^p(A;\Delta').
\end{equation}

\subsection{The action of substitution maps on HS-derivations}

Now, we recall the action of substitution maps on HS-derivations \cite[\S 6]{nar_subst_2018} and the behavior of the $\varepsilon$-derivations of Notation \ref{notacion:varepsilon-fd} on HS-derivations \cite[\S 3]{HS-vs-der}.
\medskip

Let $\bfs =\{s_1,\dots,s_p\}$, $\bft =\{t_1,\dots,t_q\}$ be sets of variables, $\Delta \subset \N^p$, $\nabla \subset \N^q$ non-empty co-ideals and let us write $R=\End_k(A)$.
\medskip

For each substitution map $\varphi:A[[\bfs]]_\Delta \to A[[\bft]]_\nabla$, we know (see \ref{nume:def-sbullet}) that $\varphi \sbullet \U^p(R;\Delta) \subset \U^q(R;\nabla)$, and in fact we have $\varphi \sbullet \HS^p_k(A;\Delta) \subset \HS^q_k(A;\nabla)$ (see \cite[Proposition 10]{nar_subst_2018}). 
\medskip

For each $i=1,\dots,p$ and each $D\in \HS^p_k(A;\Delta)$ we know that (see Notation \ref{notacion:varepsilon-fd} and \cite[Proposition 3.1.2]{HS-vs-der}):
$$ \varepsilon^i(D), \varepsilon(D), \overline{\varepsilon}^i(D), \overline{\varepsilon}(D) \in \Der_k(A)[[\bfs]]_{\Delta,+}=\Der_k(A)[[\bfs]]_\Delta \cap R[[\bfs]]_{\Delta,+},
$$
and that the map $\xi: R[[\bfs]]_{\Delta,+} \to \U^{p+1}(R;\Delta \times \{0,1\})$ defined in Lemma \ref{lemma:xi} gives rise to a injective group homomorphism 
\begin{equation}   \label{eq:xi-HS}      
 \xi: \Der_k(A)[[\bfs]]_{\Delta,+} \longrightarrow \HS^{p+1}_k(A;\Delta \times \{0,1\})
\end{equation}
 whose image is the set of $D\in \HS^{p+1}_k(A;\Delta \times \{0,1\})$ such that $\supp D \subset \{(0,0)\} \cup ((\Delta \setminus \{0\}) \times \{1\})$.
\medskip

If we denote $\text{\rm D}^p_k(A;\Delta) := \HH^p(R;\Delta) \bigcap \left(\Der_k(A[[\bfs]]_{\Delta,+}\right)^p$,
Theorem 3.1.6 of \cite{HS-vs-der} tells us that, whenever $\QQ\subset k$, the 
diagram in Proposition \ref{prop:varepsilon-bijective-Q} induces a commutative diagram with bijective maps:
\begin{equation}
\begin{tikzcd}   \label{eq:varepsilon-HS-Q-bijective}   
\HS^p_k(A;\Delta)  
 \ar[r,"\bepsilon","\simeq"'] \ar[dr,"\varepsilon","\simeq"'] & \text{\rm D}^p_k(A;\Delta)  \ar[d,"\bSigma", "\simeq"'] 
 \\
 & \Der_k(A)[[\bfs]]_{\Delta,+}.
\end{tikzcd}
\end{equation}

\begin{defi} \label{def:D-element}
 Let $S$ be a $k$-algebra over $A$, $D\in\HS^p_k(A;\Delta)$ and $r\in\U^p(S;\Delta)$. We say that
 $r$ is a {\em $D$-element}  if 
$ r a = \widetilde{D}(a) r$ for all $a\in A[[\bfs]]_\Delta$. 
\end{defi}

For the ease of the reader, we include the following result (see \cite[Theorem 3.2.5]{HS-vs-der}).

\begin{thm} \label{thm:epsilon-bullet}
For  substitution map $\varphi: A[[\bfs]]_\Delta \to A[[\bft]]_\nabla$ and each HS-derivation $D\in \HS^p_k(A;\Delta)$, there exists a family 
$$\left\{ \left. {\bf N}^{j,i}_{e,h}\ \right\vert\ 1\leq j\leq q, 1\leq i \leq p, e\in \nabla, h \in \Delta, |h|\leq |e|\right\} \subset A
$$ 
such that for any $k$-algebra $S$ over $A$ and any $D$-element $r\in \U^p(S;\Delta)$, we have:
$$
\varepsilon^j_e(\varphi\sbullet r) = \sum_{\substack{\scriptscriptstyle   0<|h|\leq |e|\\ \scriptscriptstyle i\in \supp h}}
{\bf N}^{j,i}_{e,h}
\varepsilon^i_h(r)\quad \forall e\in \nabla,\ \forall j=1,\dots,q. 
$$
\end{thm}

\section{Main results}

\subsection{The integrable connection associated with a HS--module}

Throughout this section $k$ will be a commutative ring and $A$ a commutative $k$-algebra. 
First we recall the notions of HS-structure and HS-module (see \cite[\S 3.1]{nar_envelop}).

\begin{defi} \label{def:HS-structure}
Let $R$ be a $k$-algebra over $A$. A {\em HS-structure}  on $R$ over $A/k$ is a system of maps
$ \Uppsi= \left\{\Uppsi^p_\Delta: \HS^p_k(A;\Delta) \longrightarrow \U^p(R;\Delta),\ p\in \N, \Delta\in\coide{\N^p}\right\}
$
such that\footnote{Actually, from (\ref{eq:U-inv-limit-finite}) and (\ref{eq:HS-inv-limit-finite}) we could restrict ourselves to non-empty \underline{finite} co-ideals.}:
\begin{enumerate}
\item[(i)] The $\Uppsi^p_\Delta$ are group homomorphisms.
\item[(ii)] (Leibniz rule) For any $D\in \HS^p_k(A;\Delta)$, $\Uppsi^p_\Delta(D)$ is a $D$-element (see Definition \ref{def:D-element}), i.e. $\Uppsi^p_\Delta(D)\, a = \widetilde{D}(a) \Uppsi^p_\Delta(D) $ for all $a\in A$.
\item[(iii)] For any substitution map $\varphi\in\Sub_A(p,q;\Delta,\nabla)$ and for any $D\in \HS^p_k(A;\Delta)$ we have $\Uppsi^q_\nabla(\varphi\sbullet D) = \varphi\sbullet \Uppsi^p_\Delta(D)$.
\end{enumerate}
If $R'$ is another $k$-algebra over $A$ and $f:R\to R'$ is a map of $k$-algebras over $A$, then any HS-structure $\Uppsi$ on $R$ over $A/k$ gives rise to a HS-structure $f\pcirc \Uppsi$ on $R'$ over $A/k$ defined as
$$\left( f\pcirc \Uppsi \right)^p_\Delta := \overline{f} \pcirc \Uppsi^p_\Delta,\quad p\in \N, \Delta\in\coide{\N^p}.
$$
If $R$ is filtered, we will say that a HS-structure $\Uppsi$ on $R$ over $A/k$ is {\em filtered}  if 
$$\Uppsi^p_\Delta(\HS^p_k(A;\Delta) ) \subset \Ufil^p(R;\Delta)\quad \forall p\in\N,\ \forall \Delta\in\coide{\N^p}.
$$
If $ \Uppsi$ is a {\em HS-structure} on $R$ over $A/k$, $\alpha\in \N^p$ and $\Delta = \{\alpha'\in\N^p\ |\ \alpha'\leq \alpha\}$, we will simply denote $\Uppsi^p_\alpha := \Uppsi^p_\Delta$.
\end{defi}

\begin{exam} \label{exam:tauto-HS-Diff}
The inclusions 
$\HS^p_k(A;\Delta) \subset \U^p(\DD_{A/k};\Delta) \subset \U^p(\End_k(A);\Delta)
$ give rise to the ``tautological'' HS-structures on $\DD_{A/k}$ and on $\End_k(A)$ over $A/k$, which are obviously filtered.
\end{exam}

\begin{defi}
(1) 
A {\em left HS-module} (resp. a  {\em right HS-module}) over $A/k$ is an $A$-module $E$ endowed with a HS-structure on $\End_k(E)$ (resp. on $\End_k(E)^{\text{\rm opp}}$) over $A/k$.
\smallskip

\noindent (2) A HS-map from a left (resp. a right) HS-module $(E,\Phi)$ to a left (resp. to a right) HS-module $(F,\Uppsi)$ is an $A$-linear map $f:E\to F$ such that $\overline{f} \pcirc \Phi^p_\Delta(D) = \Uppsi^p_\Delta(D) \pcirc \overline{f}$
for all $p\in\N$,  for all $\Delta\in\coide{\N^p}$, for all $\alpha\in\Delta$ and for all 
$D\in \HS^p_k(A;\Delta)$.
\end{defi}

The notions of HS-structure and HS-module are inspired by the notions of ``admissible map'' of a Lie--Rinehart algebra (cf. \cite[\S 2]{rine-63} and \cite[\S 2]{MR2000f:53109})  and of integrable connection. Let us recall (a convenient version of) these notions.

\begin{defi} \label{defi:LR-admi}
Let $R$ be a $k$-algebra over $A$. We say that a map $\nabla: \Der_k(A) \to R$ is {\em LR-admissible} (LR for Lie--Rinehart) if the following conditions hold:
\begin{enumerate}
\item[i)] $\nabla$ is left $A$-linear.
\item[ii)] (Leibniz rule) $\nabla(\delta) a = a \nabla(\delta) + \delta(a) 1_R$ for all $\delta \in \Der_k(A)$ and all $a\in A$.
\item[iii)] $\nabla([\delta,\delta']) = [\nabla(\delta),\nabla(\delta')]$ for all $\delta,\delta' \in \Der_k(A)$.
\end{enumerate}
\end{defi}

\begin{defi} A left (resp. right) {\em integrable connection} on an $A$-module $E$ over $A/k$ is a LR-admissible map $\nabla: \Der_k(A) \to \End_k(E)$ (resp. $\nabla: \Der_k(A) \to \End_k(E)^{\text{\rm opp}}$).
\end{defi}

\begin{nota} The above definition differs slightly from J.L. Koszul's one as presented in \cite[Definitions 2.4 and 2.14]{del_70}. Both definitions coincide whenever the $A$-module of differential forms $\Omega_{A/k}$ is projective of finite rank.
\end{nota}

The goal of this section is to show that any HS-structure on $R$ over $A/k$ gives rise to a natural LR-admissible map $\Der_k(A) \to R$, and consequently, that any HS-module over $A/k$ carries a natural integrable connection.
\medskip

Let us notice that for any $k$-algebra $R$ over $A$, we may identify the groups $(R,+)$ and $\U(R;1)$ through the natural group isomorphism
\begin{equation} \label{eq:R-U(R;1)}
 r \in R \longmapsto (1 + rs) \in \U(R;1).
\end{equation}
Moreover, this map translates the $(A;A)$-bimodule structure on $R$ into the action of substitution maps in $\Sub_A(1,1;\{0,1\},\{0,1\}) \equiv A$. Namely, for each $a\in A$ and each $r\in R$, we have:
\begin{eqnarray*}
&
ar \in R \longmapsto (1+ars) = a \sbullet (1+rs) \in \U(R;1),
&
\\
&
ra \in R \longmapsto (1+ras) = (1+rs)\sbullet a \in \U(R;1).
\end{eqnarray*}
In the same vein we know that the map
\begin{equation} \label{eq:Der-HS1}
 \delta \in \Der_k(A) \longmapsto (1 + \delta s) \in \HS_k(A;1)
\end{equation}
is an isomorphism of groups, where we are considering the addition as internal operation in $\Der_k(A)$. Moreover, this map also translates the left $A$-module structure on $\Der_k(A)$ into the left action of substitution maps in $\Sub_A(1,1;\{0,1\},\{0,1\}) \equiv A$.
\medskip

Assume that $ \Uppsi= \left\{\Uppsi^p_\Delta: \HS^p_k(A;\Delta) \longrightarrow \U^p(R;\Delta),\ p\in \N, \Delta\in\coide{\N^p}\right\}$ is a HS-struc\-ture on $R$ over $A/k$ and let us denote by $\nabla: \Der_k(A) \to R$ the homomorphism of additive groups defined by the following commutative diagram:
$$
\begin{CD}
\Der_k(A) @>{\nabla}>> R\\
@V{\text{(\ref{eq:Der-HS1})}}V{\simeq}V @V{\simeq}V{\text{(\ref{eq:R-U(R;1)})}}V\\
\HS_k(A;1) @>{\Uppsi^1_1}>> \U(R;1).
\end{CD}
$$
Explicitly: $\Uppsi^1_1(\Id+\delta s) = 1 + \nabla(\delta) s$.

\begin{prop} \label{prop:HS-struc-LR-admi}
With the above notations, the map  $\nabla: \Der_k(A) \to R$ is LR-admissible.
\end{prop}

\begin{proof} We need to prove properties i), ii), iii) in Definition \ref{defi:LR-admi}. Clearly, Property i) comes from property (iii) in Definition \ref{def:HS-structure} and property ii) comes from property (ii) in Definition \ref{def:HS-structure}.
\medskip

To prove property iii),
let us consider the substitution maps $\iota:A[[s]]_1 \to A[[s,s']]_{(1,1)}$, $\iota':A[[s]]_1 \to A[[s,s']]_{(1,1)}$ and $\varphi:A[[s]]_1 \to A[[s,s']]_{(1,1)}$ given by $\iota(s) = s$, $\iota'(s)= s'$ and $\varphi(s)=s s'$, 
and let us write $u:= \nabla(\delta)$, $u':= \nabla(\delta')$, $u'':= \nabla([\delta,\delta'])$, 
$v:= 1+us$, $v':= 1+u's$, $v'':= 1+u''s$, $w:= \Id + \delta s$, $w':= \Id + \delta' s$ and $w'':= \Id + [\delta,\delta'] s$. We have
$\Uppsi^1_1(w) = v$, $\Uppsi^1_1(w') = v'$ and
$\Uppsi^1_1(w'') = v''$, and since $\Uppsi$ is compatible with the action of substitution maps, we have:
\begin{eqnarray*}
& 1 + u'' s s' = \varphi \sbullet v'' = \varphi \sbullet \Uppsi^1_1(\Id + [\delta,\delta'] s) =
\Uppsi^2_{(1,1)}\left(\varphi \sbullet (\Id + [\delta,\delta'] s)\right) \stackrel{(\star)}{=} &
\\
&
\Uppsi^2_{(1,1)}\left( \left(\iota\sbullet w \right) \left( \iota'\sbullet w' \right) \left(\iota\sbullet w\right)^* \left( \iota'\sbullet w' \right)^*  \right) =
&
\\
&
\Uppsi^2_{(1,1)}\left(\iota\sbullet w \right) 
\Uppsi^2_{(1,1)}\left( \iota'\sbullet w' \right) 
\Uppsi^2_{(1,1)}\left(\left(\iota\sbullet w\right)^* \right) 
\Uppsi^2_{(1,1)}\left(\left( \iota'\sbullet w' \right)^* \right) =
&
\\
&
\left(\iota\sbullet \Uppsi^1_1 (w) \right) 
\left(\iota'\sbullet \Uppsi^1_1 (w') \right) 
\left(\iota\sbullet \Uppsi^1_1 (w) \right)^*
\left(\iota'\sbullet \Uppsi^1_1 (w') \right)^* =
&
\\
&
\left(\iota\sbullet v \right) 
\left(\iota'\sbullet  v' \right) 
\left(\iota\sbullet  v \right)^*
\left(\iota'\sbullet  v' \right)^* \stackrel{(\star)}{=} \varphi\sbullet (1+[u,u']s) = 1+[u,u']ss',
&
\end{eqnarray*}
where the $(\star)$ comes from Lemma \ref{lema:aux-brackets},
and so $\nabla([\delta,\delta']) = u'' = [u,u'] = [\nabla(\delta),\nabla(\delta')]$.
\end{proof}

\begin{cor} \label{cor:HS-mod->IC}
Any left (resp. right) HS-module $(E,\Uppsi)$ over $A/k$ carries a natural left (resp. right) integrable connection $\nabla: \Der_k(A) \to \End_k(E)$ (resp. $\nabla: \Der_k(A) \to \End_k(E)^{\text{\rm opp}}$) given by:
$$  \Uppsi^1_1(\Id+\delta s) = 1 + \nabla(\delta) s,\quad \forall \delta\in \Der_k(A).
$$
\end{cor}

LR-admissible map $\nabla: \Der_k(A) \to R$ in Proposition \ref{prop:HS-struc-LR-admi} satisfies a remarkable compatibility with respect to the maps $\varepsilon^i,\varepsilon,\overline{\varepsilon}^i,\overline{\varepsilon}:\U^p(R;\Delta) \to R[[\bfs]]_{\Delta,+}$  (see Notation \ref{notacion:varepsilon-fd}).

\begin{prop} \label{prop:compatib-nabla-epsilon}
Under the above hypotheses, for each integer $p\geq 1$, for each $\Delta \subset \N^p$ non-empty co-ideal and for each $i=1,\dots,p$, the following diagram is commutative:
$$
\begin{tikzcd}
\HS^p_k(A;\Delta) \ar[r,"\Uppsi^p_\Delta"] \ar[d,"\varepsilon^i"'] & \U^p(R;\Delta) \ar[d,"\varepsilon^i"]\\
\Der_k(A)[[\bfs]]_{\Delta,+} \ar[r,"\overline{\nabla}"] & R[[\bfs]]_{\Delta,+},
\end{tikzcd}
$$
where $\bfs = \{s_1,\dots,s_p\}$ and $\overline{\nabla}$ is the obvious map induced by $\nabla$.
\end{prop}

\begin{proof} Let us call $\sigma^i: A[[\bfs]]_\Delta \to A[[\bfs,\tau]]_{\Delta \times \{0,1\}}$ the substitution map given by $\sigma^i(s_j) = s_j$ if $j\neq i$ and $\sigma^i(s_i) = s_i + s_i \tau$, and $\Delta' = 
\Delta \times \{0,1\} \subset \N^{p+1}$.
\medskip

By using the injective map $\xi$ (see Lemma \ref{lemma:xi} and (\ref{eq:xi-HS})) and Proposition  \ref{prop:xi-varepsilon}, it is enough to prove the commutativity of the two following diagrams:
$$
\begin{tikzcd}
\HS^p_k(A;\Delta) \ar[r,"\Uppsi^p_\Delta"] \ar[d,"\xi \pcirc \varepsilon^i"'] & \U^p(R;\Delta) \ar[d,"\xi \pcirc\varepsilon^i"]\\
\HS^{p+1}_k(A;\Delta') \ar[r,"\Uppsi^{p+1}_{\Delta'}"'] & \U^{p+1}(R;\Delta')
\end{tikzcd}
$$
and
$$
\begin{tikzcd}
\Der_k(A)[[\bfs]]_{\Delta,+} \ar[r,"\overline{\nabla}"] \ar[d,"\xi"'] & R[[\bfs]]_{\Delta,+} \ar[d,"\xi"] \\
\HS^{p+1}_k(A;\Delta') \ar[r,"\Uppsi^{p+1}_{\Delta'}"'] & \U^{p+1}(R;\Delta').
\end{tikzcd}
$$
The commutativity of the first diagram is clear from properties (i) and (iii) in Definition \ref{def:HS-structure} and Proposition \ref{prop:xi-varepsilon}:
$$ \Uppsi^{p+1}_{\Delta'} (\xi(\varepsilon^i(D))) = \Uppsi^{p+1}_{\Delta'} (D^* (\sigma^i \sbullet D)) = 
\Uppsi^{p}_{\Delta} (D)^* (\sigma^i \sbullet \Uppsi^{p}_{\Delta} (D) ) = \xi(\varepsilon^i(\Uppsi^{p}_{\Delta} (D) )).
$$
For the commutativity of the second diagram, since all the involved maps are compatible with truncations and that any element in $\U^{p+1}(R;\Delta')$ is determined by its truncations to the $\Omega' = \Omega \times \{0,1\}$, with $\Omega \subset \Delta$ a non-empty finite co-ideal, we may assume that $\Delta$ is finite. In this case we have:
\begin{eqnarray*}
&\displaystyle 
\xi \left( \overline{\nabla} \left( \sum_{\scriptscriptstyle \alpha\in \Delta^*} \delta_\alpha \bfs^\alpha \right) \right) = \xi \left(\sum_{\scriptscriptstyle \alpha\in \Delta^*} \nabla(\delta_\alpha) \bfs^\alpha \right) =
1_R + \sum_{\scriptscriptstyle \alpha\in \Delta^*} \nabla(\delta_\alpha) \bfs^\alpha \tau =
&
\\
& \displaystyle 
\prod_{\scriptscriptstyle \alpha\in \Delta^*} (1_R+ \nabla(\delta_\alpha) \bfs^\alpha \tau)
\end{eqnarray*}
and
\begin{eqnarray*}
&\displaystyle 
\Uppsi^{p+1}_{\Delta'} \left( \xi \left( \sum_{\scriptscriptstyle \alpha\in \Delta^*} \delta_\alpha \bfs^\alpha \right) \right) = \Uppsi^{p+1}_{\Delta'} \left( \Id + \sum_{\scriptscriptstyle \alpha\in \Delta^*} \delta_\alpha \bfs^\alpha \tau \right) = 
&
\\
&\displaystyle 
\Uppsi^{p+1}_{\Delta'} \left( 
\prod_{\scriptscriptstyle \alpha\in \Delta^*} (\Id+ \delta_\alpha \bfs^\alpha \tau) \right) =
\prod_{\scriptscriptstyle \alpha\in \Delta^*} \Uppsi^{p+1}_{\Delta'} (\Id+ \delta_\alpha \bfs^\alpha \tau),
\end{eqnarray*}
where $\Delta^* = \Delta \setminus \{0\}$, and so it is enough to prove that:
$$
\Uppsi^{p+1}_{\Delta'} (\Id+ \delta \bfs^\alpha \tau) = 1_R+ \nabla(\delta) \bfs^\alpha \tau
$$
for all $\alpha \in \Delta^*$ and all $\delta\in \Der_k(A)$.
\medskip

Let $\varphi: A[[t]]_1 \to A[[\bfs,\tau]]_{\Delta'}$ be the substitution map given by $\varphi(t) = \bfs^\alpha \tau$. We have:
\begin{eqnarray*}
&
\Uppsi^{p+1}_{\Delta'} (\Id+ \delta \bfs^\alpha \tau) = \Uppsi^{p+1}_{\Delta'} (\varphi \sbullet (\Id + \delta t)) = \varphi \sbullet \Uppsi^1_1 (\Id + \delta t) =
&
\\
& \varphi \sbullet (1_R+ \nabla(\delta) t) = 1_R + \nabla(\delta) \bfs^\alpha \tau
\end{eqnarray*}
and we are done.
\end{proof}

\begin{cor} \label{cor:compatib-nabla-epsilon}  
Under the above hypotheses, for each integer $p\geq 1$ and for each non-empty co-ideal $\Delta \subset \N^p$, the following diagram is commutative:
$$
\begin{tikzcd}
\HS^p_k(A;\Delta) \ar[r,"\Uppsi^p_\Delta"] \ar[d,"\varepsilon"'] & \U^p(R;\Delta) \ar[d,"\varepsilon"]\\
\Der_k(A)[[\bfs]]_{\Delta,+} \ar[r,"\overline{\nabla}"] & R[[\bfs]]_{\Delta,+},
\end{tikzcd}
$$
where $\bfs = \{s_1,\dots,s_p\}$ and $\overline{\nabla}$ is the obvious map induced by $\nabla$.
\end{cor}

\begin{proof} It is a straightforward consequence of the fact that $\varepsilon = \sum_{i=1}^p \varepsilon^i$.
\end{proof}

\begin{nota} 
Similar results to Proposition  \ref{prop:compatib-nabla-epsilon} and Corollary  \ref{cor:compatib-nabla-epsilon} hold for $\overline{\varepsilon}^i$ and $\overline{\varepsilon}$ instead of $\varepsilon^i$ and $\varepsilon$.
\end{nota}

\subsection{HS-enveloping algebras versus LR-enveloping algebras}

In this section, $k$ will be a commutative ring and $A$ a commutative $k$-algebra.
\medskip

First, we recall the notion of the enveloping algebra of Hasse--Schmidt derivations introduced in \cite[\S 3.3]{nar_envelop}.

\begin{prop} \label{prop:HS-struc-dU} (see Proposition 3.3.5 in {\em loc.~cit.}) 
There is a filtered $k$-algebra $\dU_{A/k}^{\text{\rm\tiny HS}}$ over $A$ endowed with a universal HS-structure $\Upupsilon$ over $A/k$, i.e. for any $k$-algebra $R$ over $A$ and any HS-structure $\Uppsi$ on $R$ over $A/k$, there is a unique map $f:\dU_{A/k}^{\text{\rm\tiny HS}} \to R$ of $k$-algebras over $A$ such that $f\pcirc \Upupsilon = \Uppsi$. Moreover, $\Upupsilon$ is a filtered HS-structure.
\end{prop}

The algebra $\dU_{A/k}^{\text{\rm\tiny HS}}$ is called the {\em enveloping algebra} of the Hasse--Schmidt derivations of $A$ over $k$. It generalizes the enveloping algebra of the Lie--Rinehart algebra $\Der_k(A)$, that now we recall.

\begin{prop} (see \cite[\S 2]{rine-63}) There is a filtered $k$-algebra $\dU_{A/k}^{\text{\rm\tiny LR}}$ over $A$ endowed with a universal LR-admissible map $\sigma: \Der_k(A) \to \dU_{A/k}^{\text{\rm\tiny LR}}$, i.e.
 for any $k$-algebra $R$ over $A$ and any  LR-admissible map $\uppsi: \Der_k(A) \to R $, there is a unique map $f:\dU_{A/k}^{\text{\rm\tiny LR}} \to R$ of $k$-algebras over $A$ such that $f\pcirc \sigma = \uppsi$. Moreover,  its graded ring is commutative and $\sigma$ induces a canonical map of graded $A$-algebras $\Sim_A \Der_k(A) \to \gr \dU_{A/k}^{\text{\rm\tiny LR}}$.
\end{prop}

We deduce the existence of a unique map $\bupupsilon^{\text{\rm\tiny LR}}:\dU_{A/k}^{\text{\rm\tiny LR}} \longrightarrow \DD_{A/k}$ of filtered $k$-algebras over $A$ such that the following diagram is commutative:
\begin{equation*}  
\begin{tikzcd}
\dU_{A/k}^{\text{\rm\tiny LR}}  \ar[r,"\bupupsilon^{\text{\rm\tiny LR}}"]   &   \DD_{A/k} \\
&  \Der_k(A).\ar[ul,"\sigma"] \ar[u,"\text{\rm incl.}"'] 
\end{tikzcd}
\end{equation*}

\begin{defi} (Cf. \cite{brown_1978,mat-intder-I,nar_2012})  \label{def:HS-integ}
Let $m\geq 1$ be an integer or $m=\infty$, and $\delta:A\to A$ a $k$-derivation.
 We say that $\delta$ is {\em $m$-integrable} (over $k$)  if there is a
HS-derivation $D\in \HS_k(A;m)$ such that $D_1=\delta$. A such $D$ is called a $m$-integral of $\delta$. The set of $m$-integrable
$k$-derivations of $A$ is denoted by $\Ider_k(A;m)$. 
We say that $\delta$ is
{\em f-integrable} {\em (finite integrable)} if it is $m$-integrable for all \underline{integers} $m\geq 1$. The set of f-integrable
$k$-derivations of $A$ is denoted by $\Ider^f_k(A)$. 
\end{defi}

It is clear that the $\Ider_k(A;m)$ and $\Ider^f_k(A)$ are $A$-submodules of $\Der_k(A)$, and 
if $\QQ\subset k$, any $k$-derivation of $A$ is $\infty$-integrable, and so $\Der_k(A) = \Ider^f_k(A) = \Ider_k(A;\infty)$ (cf. \cite[p. 230]{mat-intder-I}).
\bigskip

\numero \label{nume:properties-U-HS}
Let us summarize the main properties of $(\dU_{A/k}^{\text{\rm\tiny HS}},\Upupsilon)$: 
\begin{enumerate}
\item[(i)] The tautological filtered HS-structure on $\DD_{A/k}$ in Example \ref{exam:tauto-HS-Diff} induces 
a canonical map $\bupupsilon^{\text{\rm\tiny HS}}:\dU_{A/k}^{\text{\rm\tiny HS}} \longrightarrow \DD_{A/k}$ of filtered $k$-algebras over $A$ (see Proposition 3.3.3 of \cite{nar_envelop}).
\item[(ii)]  The associated graded ring $\gr \dU_{A/k}^{\text{\rm\tiny HS}}$ is commutative (see Theorem 3.3.8 of \cite{nar_envelop}).
\item[(iii)] Let $\delta:A\to A$ be a f-integrable $k$-derivation and $m\geq 1$ an integer. If $D\in \HS_k(A;m)$ is a $m$-integral of $\delta$, then the symbols $\sigma_m(D_m) \in \gr^m \DD_{A/k}$ and $\sigma_m(\Upupsilon^1_m(D)_m) \in \gr^m \dU_{A/k}^{\text{\rm\tiny HS}}$ only depend of $\delta$ and not on the particular choice of the $m$-integral $D$ (see Corollary 2.7 of \cite{nar_2009} and Corollary 3.4.2 of \cite{nar_envelop}). Let us denote $\chi^m(\delta) := \sigma_m(D_m) \in \gr^m \DD_{A/k}$ and $\bupchi^m(\delta) :=
\sigma_m(\Upupsilon^1_m(D)_m) \in \gr^m \dU_{A/k}^{\text{\rm\tiny HS}}$.
\item[(iv)] Let us denote $\Gamma_A M$ the universal power divided algebra of the $A$-module $M$ and $\gamma_m: M \to \Gamma_A M$, $m\geq 1$, the universal power divided maps (cf. \cite[Appendix A]{bert_ogus}).
There are unique maps of graded $A$-algebras 
$$\vartheta^f: \Gamma_A  \Ider^f_k(A) \to \gr \DD_{A/k},\quad \bupvartheta:\Gamma_A \Ider^f_k(A) \to \gr^m \dU_{A/k}^{\text{\rm\tiny HS}}
$$
such that $\vartheta^f \pcirc \gamma_m = \chi^m$ and $\bupvartheta \pcirc \gamma_m = \bupchi^m$ for all $m\geq 1$ (see  (2.6) in \cite{nar_2009}\footnote{Actually, the existence of $\vartheta$ in this reference is proven for 
$\Ider_k(A;\infty)$ instead of $\Ider^f_k(A)$, but the proof in the second case remains essentially the same as in the first one.}
 and Corollary 3.4.3 of \cite{nar_envelop}). Moreover, the following diagram is commutative:
\begin{equation*}  
\begin{tikzcd}
\Gamma_A \Ider^f_k(A)  \ar[r,"\bupvartheta"]  \ar[rd,"\vartheta^f"'] &   \gr \dU_{A/k} \ar[d,"\gr \bupupsilon^{\text{\rm\tiny HS}}"]\\
 &  \gr \DD_{A/k}.
\end{tikzcd}
\end{equation*}
\item[(v)] If $\Ider^f_k(A)=\Der_k(A)$, then the map $\bupvartheta:\Gamma_A \Ider^f_k(A) \longrightarrow \gr \dU_{A/k}$ is surjective (see Proposition 3.4.4 of \cite{nar_envelop}).
\end{enumerate}
\bigskip

\noindent
After Proposition  \ref{prop:HS-struc-LR-admi}, the HS-structure $\Upupsilon$  on 
$\dU_{A/k}^{\text{\rm\tiny HS}}$ over $A/k$ (see Proposition \ref{prop:HS-struc-dU}) induces a natural LR-admissible map $\nabla^{\text{\rm\tiny HS}}: \Der_k(A) \to \dU_{A/k}^{\text{\rm\tiny HS}}$ given by
$$  \Upupsilon^1_1(\Id+\delta s) = 1 + \nabla^{\text{\rm\tiny HS}}(\delta) s,\quad \forall \delta \in \Der_k(A),
$$
which in turn induces a unique map of $k$-algebras over $A$:
\begin{equation} \label{eq:bkappa}
\bkappa: \dU_{A/k}^{\text{\rm\tiny LR}} \to \dU_{A/k}^{\text{\rm\tiny HS}}
\end{equation} 
such that $\bkappa \pcirc \sigma = \nabla^{\text{\rm\tiny HS}}$, 
which is obviously filtered and $\bupupsilon^{\text{\rm\tiny HS}} \pcirc \bkappa = \bupupsilon^{\text{\rm\tiny LR}}$.
\bigskip

The goal of this section is to prove the main result of this paper, namely, if $\QQ \subset k$, then the map (\ref{eq:bkappa}) is an isomorphism.
\medskip
 
From now on, we assume that $\QQ \subset k$. 
Let $R$ be a $k$-algebra over $A$ endowed with a LR-admissible map $ \nabla : \Der_k(A) \rightarrow R$ (see Definition \ref{defi:LR-admi}).

\begin{thm} \label{thm:from-nabla-to-Uppsi}
Under the above hypotheses, there is a unique HS-structure $\Uppsi = \left\{\Uppsi^p_\Delta\right\}$ on $R$ over $A/k$ such that for each $p\geq 1$ and each non-empty co-ideal $\Delta\subset \N^p$, the following diagram is commutative:
$$
\begin{tikzcd}
\HS^p_k(A;\Delta)  \ar[r,"\Uppsi^p_\Delta"] \ar[d,"\varepsilon"',"\simeq"] &
\U^p(R;\Delta) \ar[d,"\varepsilon","\simeq"']\\
\Der_k(A)[[\bfs]]_{\Delta,+} \ar[r,"\overline{\nabla}"] & R[[\bfs]]_{\Delta,+},
\end{tikzcd}
$$
where we denote $\bfs=\{s_1,\dots,s_p\}$ and
$\overline{\nabla}:\Der_k(A)[[\bfs]]_{\Delta,+} \to R[[\bfs]]_{\Delta,+}$ the left $A[[\bfs]]_\Delta$-linear map induced by $\nabla$: 
$$\overline{\nabla} \left( \sum \delta_\alpha \bfs^\alpha\right) = \sum \nabla(\delta_\alpha) \bfs^\alpha.
$$ 
Moreover, if $R = \cup_{d\geq 0} R_d$ is filtered and $\Ima \nabla \subset R_1$, then $\Uppsi$ is a filtered HS-structure.
\end{thm}

\begin{proof}
We define $\Uppsi^p_\Delta:\HS^p_k(A;\Delta) \longrightarrow \U^p(R;\Delta)$ by forcing the diagram in the statement to be commutative. 
Remember that the vertical arrows $\varepsilon$ are bijective from Proposition \ref{prop:varepsilon-bijective-Q} and (\ref{eq:varepsilon-HS-Q-bijective}). To simplify, let us write $\Uppsi = \Uppsi^p_\Delta$.
For each $E\in \HS^p_k(A;\Delta)$ we have $\varepsilon(\Uppsi (E)) = \overline{\nabla}(\varepsilon (E))$, i.e.:
$$ \bchi(\Uppsi (E)) = \Uppsi (E) \, \overline{\nabla}(\varepsilon (E)) \Longleftrightarrow 
|\alpha| \Uppsi(E)_\alpha = \sum_{\substack{\scriptscriptstyle \beta+\gamma= \alpha\\ \scriptscriptstyle  |\beta|>0}} \Uppsi(E)_\gamma \, \nabla (\varepsilon_\beta(E))\ \forall \alpha\in\Delta.
$$
Actually, we have a bigger commutative diagram:
$$
\begin{tikzcd}
\HS^p_k(A;\Delta)  \ar[r,"\Uppsi"] \ar[d,"\bepsilon"',"\simeq"] &
\U^p(R;\Delta) \ar[d,"\bepsilon","\simeq"']\\
\text{\rm D}^p_k(A;\Delta) \ar[r,"\overline{\overline{\nabla}}"] \ar[d,"\Sigma"',"\simeq"] &  \HH^p(R;\Delta) \ar[d,"\Sigma", "\simeq"']\\
\Der_k(A)[[\bfs]]_{\Delta,+} \ar[r,"\overline{\nabla}"] & R[[\bfs]]_{\Delta,+},
\end{tikzcd}
$$
(see (\ref{eq:varepsilon-HS-Q-bijective})) with $\overline{\overline{\nabla}} \left( \{\delta^i\}_{i=1}^p \right) = 
\{\overline{\nabla}(\delta^i)\}_{i=1}^p$ 
and $\varepsilon = \Sigma \pcirc \bepsilon$. In particular, for each $E\in \HS^p_k(A;\Delta)$ and each
$i=1,\dots,p$ we have
$ \bchi^i(\Uppsi (E)) = \Uppsi (E) \, \overline{\nabla}(\varepsilon^i (E))$ and $ \bchi(\Uppsi (E)) = \Uppsi (E) \, \overline{\nabla}(\varepsilon (E))$, or equivalently:
$$
\alpha_i \Uppsi(E)_\alpha = \sum_{\substack{\scriptscriptstyle \beta+\gamma= \alpha\\ \scriptscriptstyle  |\beta_i|>0}} \Uppsi(E)_\gamma \, \nabla (\varepsilon^i_\beta(E))\ \text{\ and\ }\  
|\alpha| \Uppsi(E)_\alpha = \sum_{\substack{\scriptscriptstyle \beta+\gamma= \alpha\\ \scriptscriptstyle  |\beta|>0}} \Uppsi(E)_\gamma \, \nabla (\varepsilon_\beta(E)),
$$
for all $\alpha\in\Delta$.
\medskip

First, we will prove that the $\Uppsi$ are group homomorphisms. Let us take $D,E\in \HS^p_k(A;\Delta)$. In order to prove $\Uppsi (D\pcirc E) = \Uppsi (D)\, \Uppsi (E)$ it is enough to prove that:
\begin{equation} \label{eq:aux-nabla-Uppsi-1}
 \overline{\nabla} (\varepsilon (D\pcirc E)) =
\varepsilon \left( \Uppsi (D\pcirc E) \right) = \varepsilon \left(
\Uppsi (D)\, \Uppsi (E) \right),
\end{equation}
but we know  that (see \ref{nume:explicit-varepsilon}, (i)):
\begin{eqnarray*}
& \varepsilon (D\pcirc E) = E^*\, \varepsilon(D)\, E + \varepsilon(E),
&
\\
&
\varepsilon \left( \Uppsi (D)\, \Uppsi (E) \right) = 
\Uppsi (E)^* \, \varepsilon \left( \Uppsi (D)\right)\, \Uppsi (E) + 
\varepsilon \left( \Uppsi (E)\right) = 
\Uppsi (E)^* \,  \overline{\nabla} (\varepsilon (D))\, \Uppsi (E) + 
 \overline{\nabla} (\varepsilon (E)),
\end{eqnarray*}
and so identity (\ref{eq:aux-nabla-Uppsi-1}) is equivalent to:
\begin{equation} \label{eq:aux-nabla-Uppsi-2}
 \overline{\nabla} (E^*\, \varepsilon (D) \, E)) =
 \Uppsi(E)^* \, \overline{\nabla} (\varepsilon( D)) \,  \Uppsi(E),
\end{equation}
which is a consequence of Lemma \ref{lemma:nabla-Uppsi}\footnote{Let us notice that the fact that the $\Uppsi$ are group homomorphism only depends on $\nabla$ being a map of $\ZZ$-Lie algebras.}.
\medskip

\noindent 
Second, let us prove that $\Uppsi(D)$ is a $D$-element for each $D\in \HS_k^p(A;\Delta)$, i.e.:
$$ \Uppsi(D)_\alpha\, a =  \sum_{\scriptscriptstyle \beta+\gamma=\alpha}  D_\beta(a)\, \Uppsi(D)_\gamma\quad \forall \alpha\in\Delta,\ \forall a\in A.
$$
For $\alpha=0$ the equality being clear, we proceed by induction on $|\alpha|$:
\begin{eqnarray*}
& \displaystyle
|\alpha| \Uppsi(D)_\alpha\, a = \sum_{\substack{\scriptscriptstyle \beta+\gamma= \alpha\\ \scriptscriptstyle  |\beta|>0}} \Uppsi(D)_\gamma \, \nabla (\varepsilon_\beta(D)) \, a \stackrel{(\star)}{=}
&
\\
& \displaystyle
\sum_{\substack{\scriptscriptstyle \beta+\gamma= \alpha\\ \scriptscriptstyle  |\beta|>0}} \Uppsi(D)_\gamma \, a \, \nabla (\varepsilon_\beta(D)) +  
\sum_{\substack{\scriptscriptstyle \beta+\gamma= \alpha\\ \scriptscriptstyle  |\beta|>0}} \Uppsi(D)_\gamma \, \varepsilon_\beta(D)(a) \stackrel{\text{(IH)}}{=}
&
\\
& \displaystyle
\sum_{\substack{\scriptscriptstyle \beta+\gamma'+\gamma''= \alpha\\ \scriptscriptstyle  |\beta|>0}} D_{\gamma'}(a)\, \Uppsi(D)_{\gamma''} \, \nabla (\varepsilon_\beta(D)) +  
\sum_{\substack{\scriptscriptstyle \beta+\gamma'+\gamma''= \alpha\\ \scriptscriptstyle  |\beta|>0}}
D_{\gamma'}(\varepsilon_\beta(D)(a))\, 
 \Uppsi(D)_{\gamma''}=
 &
\\
& \displaystyle
\sum_{\substack{\scriptscriptstyle \gamma'+\mu= \alpha\\ \scriptscriptstyle  |\mu|>0}}
D_{\gamma'}(a)\, 
\left(
\sum_{\substack{\scriptscriptstyle \beta+\gamma''= \mu\\ \scriptscriptstyle  |\beta|>0}}
\Uppsi(D)_{\gamma''} \, \nabla (\varepsilon_\beta(D))
\right) +
&
\\
& \displaystyle
\sum_{\substack{\scriptscriptstyle \nu+\gamma''= \alpha\\ \scriptscriptstyle  |\nu|>0}}
\left(
\sum_{\substack{\scriptscriptstyle \beta+\gamma'= \nu\\ \scriptscriptstyle  |\beta|>0}}
(D_{\gamma'} \pcirc \varepsilon_\beta(D))(a)
\right)
\, 
 \Uppsi(D)_{\gamma''}=
\sum_{\substack{\scriptscriptstyle \gamma'+\mu= \alpha\\ \scriptscriptstyle  |\mu|>0}}
|\mu| \, D_{\gamma'}(a)\, \Uppsi(D)_\mu
 +
  &
\\
& \displaystyle
 \sum_{\substack{\scriptscriptstyle \nu+\gamma''= \alpha\\ \scriptscriptstyle  |\nu|>0}}
|\nu|\, D_\nu(a) 
\, 
 \Uppsi(D)_{\gamma''}=
 |\alpha| \sum_{\scriptscriptstyle \beta+\gamma=\alpha}  D_\beta(a)\, \Uppsi^p_\Delta(D)_\gamma.
\end{eqnarray*}
Notice that that equality $(\star)$ uses that $\nabla$ satisfies Leibniz rule.
\medskip

\noindent 
To finish, it remains to prove that for any substitution map $\varphi: A[[\bfs]]_\Delta \to A[[\bft]]_\Omega$, with $\bft=\{t_1,\dots,t_q\}$ and $\Omega \subset \N^q$ a non-empty co-ideal, and any $D\in \HS^p_k(A;\Delta)$ we have:
$$ \Uppsi^q_\Omega(\varphi \sbullet D) = \varphi \sbullet \Uppsi^p_\Delta(D).
$$
This is equivalent to $\bepsilon(\Uppsi^q_\Omega(\varphi \sbullet D)) = \bepsilon(\varphi \sbullet \Uppsi^p_\Delta(D))$, 
but we know from Theorem \ref{thm:epsilon-bullet} that:
\begin{eqnarray*} 
& \displaystyle
\varepsilon^j_e(\varphi \sbullet \Uppsi^p_\Delta(D)) = \sum_{\substack{\scriptscriptstyle   0<|h|\leq |e|\\ \scriptscriptstyle i\in \supp h}}
{\bf N}^{j,i}_{e,h}
\varepsilon^i_h(\Uppsi^p_\Delta(D)) = 
\sum_{\substack{\scriptscriptstyle   0<|h|\leq |e|\\ \scriptscriptstyle i\in \supp h}}
{\bf N}^{j,i}_{e,h}
\nabla(\varepsilon^i_h(D)) \stackrel{(\star\star)}{=}
&
\\
& \displaystyle
\nabla \left( 
\sum_{\substack{\scriptscriptstyle   0<|h|\leq |e|\\ \scriptscriptstyle i\in \supp h}}
{\bf N}^{j,i}_{e,h}
\varepsilon^i_h(D) \right) =
\nabla ( \varepsilon^j_e(\varphi \sbullet D) ) =  \varepsilon^j_e( \Uppsi^q_\Omega (\varphi \sbullet D) )
\end{eqnarray*}
for all $j=1,\dots,q$ and for all $e\in \Omega$, and we are done.
Let us notice that equality $(\star\star)$ uses that $\nabla$ is $A$-linear.
\medskip

For the last part, if $R = \cup_{d\geq 0} R_d$ is filtered and $\Ima \nabla \subset R_1$, then the image of each map
$$ \overline{\nabla}: \Der_k(A)[[\bfs]]_{\Delta,+} \longrightarrow R[[\bfs]]_{\Delta,+}
$$
is contained in $R_1[[\bfs]]_{\Delta,+}$, and it is easy to see that 
$\varepsilon^{-1} \left( R_1[[\bfs]]_{\Delta,+} \right) \subset \Ufil^p(R;\Delta)$.
\end{proof}

\begin{lemma} \label{lemma:nabla-Uppsi}
Under the hypotheses of Theorem \ref{thm:from-nabla-to-Uppsi}, for each $\delta\in \Der_k(A)[[\bfs]]_\Delta$ and each $E\in \HS^p_k(A;\Delta)$ the following identity holds:
$$ \Uppsi(E) \,  \overline{\nabla} (E^*\, \delta \, E) =  \overline{\nabla} (\delta) \, \Uppsi(E).
$$
\end{lemma}

\begin{proof} 
Since all the involved maps and operations are compatible with truncations and any series in $R[[\bfs]]_\Delta$ is determined by its finite truncations, we may assume that $\Delta$ is finite, and since both terms are $k[[\bfs]]_\Delta$-linear in $\delta$, 
we may assume $\delta\in\Der_k(A)$.
By definition of $\Uppsi$, we have:
$$ \bchi(\Uppsi (E)) = \Uppsi (E) \, \overline{\nabla}(\varepsilon (E)),\quad \text{with}\ \bchi = \sum_{i=1}^p s_i \frac{\partial}{\partial s_i}.
$$
Since the $0$-term of the series $\Uppsi(E) \,  \overline{\nabla} (E^*\, \delta \, E)$ and  $\overline{\nabla} (\delta) \, \Uppsi(E)$ coincide (they are equal to $\nabla(\delta)$) and $\QQ \subset k$, it is enough to prove that both series are solution of the differential equation:
$$ \bchi(Y) = Y \, \overline{\nabla}(\varepsilon (E)).
$$
Namely:
\begin{eqnarray*}
&\displaystyle
\bchi\left(\Uppsi(E) \,  \overline{\nabla} (E^*\, \delta \, E) \right) = \bchi(\Uppsi(E)) \,  \overline{\nabla} (E^*\, \delta \, E) +
\Uppsi(E) \,  \bchi(\overline{\nabla} (E^*\, \delta \, E) ) =
&
\\
&\displaystyle
\Uppsi(E)) \, \overline{\nabla}(\varepsilon (E)) \,  \overline{\nabla} (E^*\, \delta \, E)  + \Uppsi(E) \, \overline{\nabla} (\bchi(E^*\, \delta \, E) ) = 
&
\\
&\displaystyle
\Uppsi(E)) \, \overline{\nabla}(\varepsilon (E)) \,  \overline{\nabla} (E^*\, \delta \, E)  +  
\Uppsi(E) \, \overline{\nabla} (\bchi(E^*) \, \delta \, E  + E^*\, \delta \, \bchi(E)) \stackrel{(\star)}{=}
&
\\
&\displaystyle
\Uppsi(E)) \, \overline{\nabla}(\varepsilon (E)) \,  \overline{\nabla} (E^*\, \delta \, E)  +   
\Uppsi(E) \, \overline{\nabla}  ( - \varepsilon(E) \, E^* \, \delta \, E + E^*\, \delta \, E \, \varepsilon(E))  =
&
\\
&\displaystyle
\Uppsi(E)) \, \overline{\nabla}(\varepsilon (E)) \,  \overline{\nabla} (E^*\, \delta \, E)  +   
\Uppsi(E) \, \overline{\nabla}  ( [E^*\, \delta \, E,\varepsilon(E)])  =
&
\\
&\displaystyle
\Uppsi(E)) \, \overline{\nabla} (\varepsilon (E)) \,  \overline{\nabla} (E^*\, \delta \, E)  +   
\Uppsi(E) \, [\overline{\nabla}(E^*\, \delta \, E),\overline{\nabla}(\varepsilon(E))] = 
&
\\
&\displaystyle
\Uppsi(E)) \,  \overline{\nabla} (E^*\, \delta \, E) \, \overline{\nabla}(\varepsilon (E)),
\end{eqnarray*}
where equality $(\star)$ comes from \ref{nume:explicit-varepsilon}, (ii),
and 
$$
\bchi(\overline{\nabla}(\delta)\, \Uppsi(E)) = \overline{\nabla}(\delta)\, \bchi(\Uppsi(E)) = \overline{\nabla}(\delta)\, \Uppsi(E)\, \overline{\nabla}(\varepsilon (E)).$$
\end{proof}

\begin{thm} \label{thm:bkappa-iso}
If $\QQ \subset k$, then the map (\ref{eq:bkappa})
$$\bkappa: \dU_{A/k}^{\text{\rm\tiny LR}} \to \dU_{A/k}^{\text{\rm\tiny HS}}
$$
is an isomorphism of filtered $k$-algebras over $A$.
\end{thm}

\begin{proof} By applying Theorem \ref{thm:from-nabla-to-Uppsi} to the universal LR-admissible map
$$ \sigma:\Der_k(A) \longrightarrow \dU_{A/k}^{\text{\rm\tiny LR}},$$
there is a unique filtered HS-structure $\Uppsi^{\text{\rm\tiny LR}}$ on $\dU_{A/k}^{\text{\rm\tiny LR}}$ over $A/k$ such that $\overline{\sigma} \pcirc \varepsilon=\varepsilon \pcirc \left(\Uppsi^{\text{\rm\tiny LR}}\right)^p_\Delta$ for each $p\geq 1$ and each non-empty co-ideal $\Delta \subset \N^p$, 
 and so, by Proposition \ref{prop:HS-struc-dU}, there is a unique map $ \blambda: \dU_{A/k}^{\text{\rm\tiny HS}} \longrightarrow \dU_{A/k}^{\text{\rm\tiny LR}}$ of filtered $k$-algebras over $A$ such that $\Uppsi^{\text{\rm\tiny LR}} = \blambda \pcirc \Upupsilon$.
\medskip

Let us prove that $\blambda$ is the inverse map of $\bkappa$. 
For each $\delta \in \Der_k(A)$ we have:
\begin{eqnarray*}
&\sigma(\delta) s = \overline{\sigma}(\delta s) = \overline{\sigma}(\varepsilon(\Id + \delta s)) = \varepsilon \left(\left(\Uppsi^{\text{\rm\tiny LR}}\right)^1_1(\Id + \delta s)\right) = 
\varepsilon(\overline{\blambda}(\Upupsilon^1_1(\Id + \delta s))) = 
&
\\
& 
\varepsilon(\overline{\blambda}(1 + \nabla^{\text{\rm\tiny HS}}(\delta) s)) =
\varepsilon(1+ \blambda(\nabla^{\text{\rm\tiny HS}}(\delta))s) = (\blambda \pcirc \nabla^{\text{\rm\tiny HS}})(\delta)s.
\end{eqnarray*}
So, $\sigma = \blambda \pcirc \nabla^{\text{\rm\tiny HS}} = \blambda \pcirc \bkappa \pcirc \sigma$ and we deduce 
that $\blambda \pcirc \bkappa = \Id$. 
\medskip

Since $\QQ\subset k$, we have $\Ider^f_k(A)=\Der_k(A)$  and so the map 
$\bupvartheta:\Gamma_A \Ider^f_k(A) \to \gr \dU_{A/k}^{\text{\rm\tiny HS}}$ is surjective (see \ref{nume:properties-U-HS}, (v)). We easily check that the following diagram is commutative:
$$
\begin{tikzcd}
\gr \dU_{A/k}^{\text{\rm\tiny LR}} \ar[r,"\gr \bkappa"] & \gr \dU_{A/k}^{\text{\rm\tiny HS}} \\
\Sim_A \Der_k(A) \ar[u,"\text{nat.}"] \ar[r,"\text{nat.}"] & \Gamma_A \Der_k(A), \ar[u,"\bupvartheta"']
\end{tikzcd}
$$
and since $\QQ\subset k$, we have $\Sim_A \Der_k(A) \xrightarrow{\sim} \Gamma_A \Der_k(A)$
and we deduce that $\gr \bkappa$ is surjective, and so $\bkappa$ is surjective too. 
We conclude that $\blambda$ is the the inverse map of $\bkappa$.
\end{proof}

\begin{cor} \label{cor:main}
Under the above hypotheses, the category of left (resp. right) HS-modules over $A/k$ coincide with the category of $A$-modules endowed with a left (resp. right) integrable connection over $A/k$.
\end{cor}

\bigskip

\noindent {\small \href{http://personal.us.es/narvaez/}{Luis Narv\'aez Macarro}\\
\noindent \href{http://departamento.us.es/da/}{Departamento de \'Algebra} \&\
\href{http://www.imus.us.es}{Instituto de Matem\'aticas (IMUS)}\\
\href{http://matematicas.us.es}{Facultad de Matem\'aticas}, \href{http://www.us.es}{Universidad de Sevilla}\\
Calle Tarfia s/n, 41012  Sevilla, Spain} \\
{\small {\it E-mail}\ : narvaez@us.es}

\end{document}